\documentclass[11pt]{article}

\usepackage[margin=1in]{geometry}
\usepackage{setspace}
\usepackage{amsmath,amssymb,amsthm,mathtools}
\usepackage{mathrsfs}
\usepackage{graphicx}
\usepackage{booktabs}
\usepackage{array}
\usepackage{placeins}

\usepackage[round,authoryear]{natbib}
\usepackage[hidelinks]{hyperref}
\hypersetup{
  pdftitle={Tractable Relaxations of Multivariate Stochastic Dominance via Optimal Transport and CVaR},
  pdfauthor={Bar Light}
}

\theoremstyle{plain}
\newtheorem{theorem}{Theorem}[section]
\newtheorem{corollary}[theorem]{Corollary}
\newtheorem{proposition}[theorem]{Proposition}
\newtheorem{example}[theorem]{Example}
\newtheorem{definition}[theorem]{Definition}
\newtheorem{lemma}[theorem]{Lemma}

\theoremstyle{remark}
\newtheorem{remark}[theorem]{Remark}

\newcommand{\CVaR}{\operatorname{CVaR}}

\begin{document}

\title{Tractable Relaxations of Multivariate Stochastic Dominance via Optimal Transport and CVaR}
\author{Bar Light\protect\thanks{Business School and Institute of Operations Research and Analytics, National University of Singapore, Singapore.\ e-mail: \textsf{barlight@nus.edu.sg} }}
\maketitle

\thispagestyle{empty}

\begin{abstract}
Many operational decisions involve alternatives with several uncertain
attributes. When comparing such alternatives, stochastic dominance requires
every decision maker in a prescribed class to prefer one alternative to the
other. These orders have two known limitations: a single extreme decision maker can
rule out dominance, and multivariate dominance can be difficult to verify. To
address these limitations, we develop two relaxations of their standard
representations: compensated stochastic dominance (CSD) relaxes Strassen's
coupling representation, while reference-weighted stochastic dominance (RWSD)
relaxes the integral representation. We characterize the utility class
generated by CSD, recovering known almost stochastic dominance orders as
special cases, and characterize it for RWSD in several cases. Importantly, we show that CSD
and RWSD can each be checked by computing a single value, an optimal-transport
value for CSD and a CVaR value for RWSD, and verifying that this value is
nonpositive. This inequality characterization allows us to derive finite-sample
guarantees for both orders, even when the outcome dimension is large. An application to Olist e-commerce data shows that, even when empirical FOSD fails, CSD and RWSD quantify how small this failure is and what is needed to overcome it. 
\end{abstract}

\newpage

\section{Introduction}
\label{Sec:Introduction}

Operational alternatives often affect several uncertain outcomes jointly. Expected utility provides a coherent decision rule, but specifying a multivariate utility function can be difficult to justify, and the resulting ranking may be sensitive to misspecification. Stochastic dominance offers a robust alternative: it ranks two distributions only when every utility in a prescribed class agrees with the ranking \citep{MullerStoyan2002,ShakedShanthikumar2007}. The same robustness makes exact dominance an all-or-nothing criterion. A single utility that prefers the other distribution rules out dominance, whether that disagreement is isolated or widespread. The conclusion is therefore the same when a comparison narrowly misses dominance and when it fails by a large margin. Moreover, failure does not reveal what causes the disagreement or how much the dominance requirement must be relaxed to recover the ranking.

Strassen's theorem makes this limitation especially clear for multivariate first-order stochastic dominance (FOSD). It states that $Q$ dominates $P$ if the two distributions can be coupled so that the outcome from $Q$ is at least as large as the outcome from $P$ in every coordinate almost surely \citep{Strassen1965}. For finitely supported distributions, checking whether such a coupling exists can be formulated as a transportation problem. However, the requirement is exact: under the coupling, $Q$ cannot be lower than $P$ in any coordinate with positive probability. Thus, if every coupling contains even a small coordinate decrease, FOSD fails. This can be especially restrictive for  empirical distributions, where finite or noisy data can create small reversals. This binary decision rule limits the usefulness of FOSD in applications by treating small failures identically to large ones, without identifying which coordinates cause the failure, measuring its magnitude, or indicating how close the distributions are to dominance.

The integral representation of stochastic orders creates a closely related difficulty. A family of test functions (functions that assign a numerical value to each outcome) defines a stochastic order by requiring each function to have at least as high an expected value under $Q$ as under $P$; when all of these inequalities hold, $Q$ dominates $P$ \citep{Muller1997StochasticOrders}. For example, for FOSD, a standard family of test functions consists  of indicators of upper sets. Lower-orthant dominance uses test functions based on lower sets defined by coordinatewise thresholds, whereas increasing-concave, or second-order stochastic dominance (SOSD), uses increasing concave test functions. If even one test function has a lower expected value under $Q$ than under $P$, the corresponding dominance order fails. In this paper, we use the coupling and integral representations to relax stochastic dominance in two different ways and quantify how close the distributions are to satisfying the corresponding order.

Our first relaxation, compensated stochastic dominance (CSD), relaxes the coupling representation. CSD allows selected monotone scores (functions of one or more outcomes for which improving an outcome never lowers the score) to decrease under the coupling, but limits their expected weighted decrease to a fraction of their expected weighted increase.
We call this fraction the tolerance. At zero tolerance, no expected weighted decrease is allowed; with coordinate scores, this is exactly Strassen's FOSD condition. At tolerance one, CSD requires only that the expected weighted sum of the selected scores be at least as high under $Q$ as under $P$. Hence, CSD can compare a much wider range of distributions than FOSD. 
We study general monotone scores because some applications focus on 
functions of several outcomes together, not only on  individual coordinates.
Different score choices therefore generate different CSD comparisons tailored
to the application, as illustrated in Section~\ref{Sec:Utility}. One natural
choice is to use lower-orthant scores, which we study in
Section~\ref{Sec:CSDLowerOrthantScores}.

We show that one Wasserstein optimal-transport problem using the selected
scores determines CSD for every tolerance and gives the smallest tolerance
for which dominance holds. With coordinate scores, zero tolerance corresponds
to FOSD. If FOSD fails but the weighted sum of the coordinate means is at
least as high under $Q$ as under $P$, the same optimal-transport problem gives the smallest
fraction of expected weighted coordinate decreases relative to expected
weighted coordinate increases that must be allowed for dominance. Thus,  CSD quantifies how much the exact coupling
condition must be relaxed. The optimal coupling also shows which score
decreases are compensated by score increases.

We then use transport duality to characterize the utility class that generates
CSD. For linear scores, this gives bounds on the gradients of the admissible
utilities. With coordinate scores, CSD recovers substitution multivariate
almost stochastic dominance (MASD) and gives a Wasserstein characterization
and direct computational test for the order in
\citet{MullerScarsiniTsetlinWinkler2025}. With lower-orthant scores, we show that CSD
generates a family between multivariate FOSD and second-degree lower-orthant
dominance; in one dimension, it recovers the family between FOSD and SOSD in
\citet{MullerScarsiniTsetlinWinkler2017}. Thus, our results
connect the compensation allowed in the coupling to the utility restrictions
that generate the resulting dominance order.

Our second relaxation, reference-weighted stochastic dominance (RWSD), starts from the test functions of an integral stochastic order. Each test function gives a difference between its expected value under $Q$ and under $P$, and the exact order requires every such difference to be nonnegative. RWSD uses lower and upper reference measures to set bounds on how weight can be assigned across the test functions, and dominance holds when every weighted average of the expected-value differences satisfying these bounds is nonnegative. These reference measures can therefore control how much emphasis is placed on particular attributes, tail events, or other features represented by the test functions. This construction applies to any stochastic order with an integral representation.

We show that checking RWSD reduces to a CVaR calculation. In particular, 
for finitely many test functions, checking RWSD reduces to a continuous
knapsack problem, which can be solved by sorting the expected-value
differences. RWSD can therefore be interpreted as quantifying how much weight
can be placed on these test functions before dominance fails. We develop RWSD versions of FOSD, lower-orthant dominance,
and increasing-concave dominance, and show how the reference measures restrict
marginal utilities, curvature, and interactions between attributes when such
utility interpretations are available.

Importantly, our optimal-transport and CVaR characterizations reduce both CSD
and RWSD to checking whether a single value is nonpositive. 
Because each order can be verified by a single inequality, we are able to derive finite-sample guarantees directly from empirical distributions (see Section \ref{sec:FiniteSamFORBOTHORDERS}).
 For CSD, we prove a bound that holds simultaneously for every
tolerance, so the same sample gives an upper confidence bound on the smallest
tolerance without testing a grid of tolerances. For RWSD, we prove a corresponding
finite-sample bound that holds simultaneously over every weighting permitted
by the reference measures. Once the score or test-function families are fixed, neither guarantee contains an additional term that
grows explicitly with the outcome dimension. Thus, when distributions are
observed through finite samples, both orders provide tractable computations
for dominance together with finite-sample guarantees for certifying population
dominance.

The two relaxations are useful for different modeling questions. CSD is natural when decreases in some selected performance measures may be compensated by increases in others, whereas RWSD is natural when the concern is how much emphasis may be placed on particular test-function comparisons, such as lower-tail events.  In terms of computation, the CSD transport problem grows with the support sizes of the empirical distributions, whereas the RWSD calculation grows with the number of selected test-functions. RWSD may therefore be particularly attractive in very large empirical applications. 

Lastly, Section~\ref{Sec:OlistApplication} uses Olist e-commerce data to
illustrate CSD and RWSD. We deliberately
construct two groups of sellers so that one group has much better pre-2018
fulfillment performance, making dominance especially plausible. These sellers
fulfill customer orders by handing them to carriers, and we compare the two
groups using later delivery and customer-review outcomes. The group with
better pre-2018 fulfillment performance has higher means for every service
measure we consider. Nevertheless, empirical FOSD fails because this group
performs slightly worse in a small part of the distribution of one
carrier-transit measure. CSD and RWSD show that this failure is very small and
identify what causes it. CSD further shows how much improvement in the other
service measures is needed to compensate for the decrease in carrier transit.
Thus, the application illustrates a natural difficulty with empirical multivariate FOSD: when many outcomes are compared, a small unfavorable difference in just one coordinate can rule out dominance. CSD and RWSD instead quantify how small the failure is and identify its source.

\subsection{Related Literature}
\label{Sec:RelatedLiterature}

Our paper is related to three strands of literature.

\noindent \textbf{Stochastic dominance and optimal transport}. Several papers use optimal transport to quantify departures from stochastic dominance, either by measuring violations of the order within a coupling or by measuring the distance to a distribution that satisfies exact dominance 
\citep{DelBarrioCuestaAlbertosMatran2018,
KamihigashiStachurski2020,JunovaKopa2025,DentchevaYi2025}.
A transport-based paper more closely related to CSD is
\citet{RiouxNitsureRigottiGreenewaldMroueh2024}, who also start from the
coupling characterization of multivariate FOSD but develop smooth,
 regularized optimal-transport statistics for measuring departures
from exact dominance. Independent and concurrent work by \citet{MullerWiesel2026} develops a more
abstract optimal-transport framework for almost stochastic dominance on general
Polish spaces.\footnote{In the
univariate case, \citet{Muller2024StochasticOrders} had earlier studied
robustness of almost stochastic dominance under Wasserstein perturbations.}
Their framework uses asymmetric quasi-pseudo-metric costs on general Polish
spaces and derives a corresponding dual test-function characterization. For
substitution-MASD, their multivariate specialization is equivalent to our
coordinate-score CSD characterization. Our paper focuses
additionally on tractable computation,  finite-sample guarantees,
and RWSD. 

\noindent \textbf{Utility-based relaxations of stochastic dominance}.
A popular approach relaxes stochastic dominance while retaining a
preference interpretation. \citet{LeshnoLevy2002} introduced almost
stochastic dominance, and \citet{TsetlinWinklerHuangTzeng2015} developed general results linking relaxed integral conditions to restricted utility
classes. These ideas have been used in stochastic optimization and portfolio
choice \citep{LizyayevRuszczynski2012,LuoChenJaillet2025}. Closely related
orders restrict marginal utility, curvature, or other features of the utility
class. \citet{MullerScarsiniTsetlinWinkler2017} construct a continuum between
FOSD and SOSD, while \citet{LightPerlroth2021} introduce a parametric family
of concave-utility orders that relax SOSD and higher-orders.
Our paper is also related to preference-robust optimization under incomplete
utility information. \citet{ArmbrusterDelage2015} optimize decisions over
utility classes consistent with partial preference and shape information,
while \citet{HuStepanyan2017} introduce a reference-based almost-dominance
relation that quantifies how robustly one alternative is preferred. We instead study pairwise stochastic orders derived from coupling and integral
representations.

For multivariate outcomes, \citet{TsetlinWinkler2018} extend almost stochastic
dominance to multiattribute utility functions.
\citet{MullerScarsiniTsetlinWinkler2025} characterize multivariate almost
stochastic dominance through bounds on marginal utilities and compensated
transfers. Choosing coordinate scores makes CSD exactly their
substitution-MASD order and provides an optimal-transport characterization of
that multivariate almost stochastic dominance. General monotone scores extend the same compensation logic
to joint, application-specific performance measures and allow tradeoffs that
need not be coordinatewise. In particular, we show that lower-orthant scores extend the univariate
orders of \citet{MullerScarsiniTsetlinWinkler2017} to vector outcomes. RWSD is also related, more distantly, to weighted stochastic-dominance formulations.
\citet{HuHomemDeMelloMehrotra2014} use random weights to scalarize multivariate outcomes and then apply stochastic-dominance comparisons to the resulting scalar outcomes, while \citet{Tan2015WeightedASD} introduces a weighting function into a particular almost-stochastic-dominance criterion. In contrast, RWSD restricts how weight may be distributed across the test-function comparisons that generate an integral stochastic order and requires the ranking to hold for every admissible weighting. This construction applies across integral stochastic orders and yields a simple CVaR characterization.

\noindent \textbf{Statistical inference for stochastic dominance}. The
statistical literature develops tests for the population inequalities that
define stochastic dominance, typically using asymptotic approximations,
bootstrap, or subsampling procedures
\citep{DavidsonDuclos2000,BarrettDonald2003,
LintonMaasoumiWhang2005}. Recent work extends inference to almost-dominance
measures and transport-based comparisons
\citep{BailloCarcamoMoraCorral2025,SongSun2026,
RiouxNitsureRigottiGreenewaldMroueh2024}. Our results use the new transport and CVaR characterizations to obtain explicit,
nonasymptotic guarantees for population dominance. Whereas standard inference
for stochastic dominance typically works with collections of inequalities
indexed by thresholds or test functions, our characterizations reduce CSD and
RWSD to the sign of a single optimization value. We derive concentration bounds
for the empirical quantities determining these values, so a sufficiently
negative empirical bound certifies population dominance. 

\section{Compensated Stochastic Dominance}
\label{Sec:CSD}

Let $D\subseteq\mathbb R^d$ be a nonempty Borel set. Throughout this section,
$P$ and $Q$ are Borel probability measures supported on $D$ with finite first
moments, $\mathcal C(P,Q)$ denotes their set of couplings, and
$t_+:=\max\{t,0\}$. If $\kappa\in\mathcal C(P,Q)$, then $\kappa$ is the
joint distribution of a pair $(X,Y)$ with $X\sim P$ and $Y\sim Q$; below,
$x$ and $y$ denote possible realizations of $X$ and $Y$, respectively.

As explained in the introduction, compensated stochastic dominance (CSD)
relaxes Strassen's coupling representation of FOSD. Fix positive weights
$\beta_1,\ldots,\beta_d$ and $\gamma\in[0,1]$. We call a coupling
$\kappa\in\mathcal C(P,Q)$ admissible when
\begin{equation}
\label{Eq:CoordinateCompensatedStrassen}
\int_{D\times D}\sum_{i=1}^d\beta_i(x_i-y_i)_+\,\kappa(dx,dy)
\le
\gamma
\int_{D\times D}\sum_{i=1}^d\beta_i(y_i-x_i)_+\,\kappa(dx,dy).
\end{equation}
At $\gamma=0$, $x_i\le y_i$ almost surely for every coordinate, as in
Strassen's characterization of FOSD. For $\gamma>0$, the expected weighted
sum of the coordinate shortfalls $(x_i-y_i)_+$ may be positive, but it cannot
exceed $\gamma$ times the expected weighted sum in the opposite direction.
If $x_i>y_i$ for some coordinate, dominance may still hold when the
coordinates and coupled outcomes for which $y_i>x_i$ provide enough
compensation with $\gamma$ controlling the compensation amount. 

To allow more general trade-offs beyond coordinates, let
$(\Theta,\mathscr A,\lambda)$ be a finite measure space and let
$s:\Theta\times\mathbb R^d\to\mathbb R$ be jointly measurable. For each
$\theta\in\Theta$, we call $s_\theta(x):=s(\theta,x)$ a score. We assume that
every score is coordinatewise nondecreasing,
$\int_\Theta |s_\theta(0)|\,\lambda(d\theta)<\infty$, and, for some finite
constant $\overline L_s$,
$\int_\Theta |s_\theta(x)-s_\theta(y)|\,\lambda(d\theta)
\le \overline L_s\|x-y\|_1$ for all $x,y\in\mathbb R^d$.
 Together with the finite first moments of $P$ and
$Q$, these are regularity conditions that make all expectations below finite.
Key examples include the lower-orthant indicators in
Section~\ref{Sec:CSDLowerOrthantScores} and finite families of linear scores
$s_j(x)=a^j\cdot x$ with $a^j\in\mathbb R_+^d$.
The score analogue of \eqref{Eq:CoordinateCompensatedStrassen} is
\begin{equation}
\label{Eq:ScoreCompensatedStrassen}
\int_{D\times D}\int_\Theta
(s_\theta(x)-s_\theta(y))_+
\,\lambda(d\theta)\,\kappa(dx,dy)
\le
\gamma
\int_{D\times D}\int_\Theta
(s_\theta(y)-s_\theta(x))_+
\,\lambda(d\theta)\,\kappa(dx,dy).
\end{equation}
 The measure $\lambda$  allows more weight to be placed on scores that matter more in the application. The parameter $\gamma$ again can be viewed as a compensation parameter that limits the expected weighted score
decrease to a fraction of the expected weighted score increase.

\begin{definition}[Compensated stochastic dominance]
\label{Def:ScoreCompensatedDominance}
For $\gamma\in[0,1]$, we write
$P\preceq_{s,\lambda,\gamma}^{\rm CSD}Q$ if there exists
$\kappa\in\mathcal C(P,Q)$ satisfying
\eqref{Eq:ScoreCompensatedStrassen}.
\end{definition}

The coordinate condition in
\eqref{Eq:CoordinateCompensatedStrassen} is the special case
$\Theta=\{1,\ldots,d\}$, $s_i(x)=x_i$, and
$\lambda(\{i\})=\beta_i$. At $\gamma=0$, the nonnegative score-decrease integral must vanish, so
$s_\theta(x)\le s_\theta(y)$ for
$\lambda\otimes\kappa$-almost every $(\theta,x,y)$. In addition, if for every
$x,y$, $s_\theta(x)\le s_\theta(y)$ for $\lambda$-almost every $\theta$
implies $x\le y$ then the coupling is monotone and Strassen's theorem gives
FOSD. We next show that only one transport problem checks CSD for every
$\gamma\in[0,1]$.

\subsection{Optimal Transport Characterization of CSD}
\label{sec:transport}

We will need the following notation. 
Define the weighted sum of the scores, the total absolute difference between
the selected score values of $x$ and $y$, and the score decrease when the
outcome moves from $x$ to $y$ by
\begin{align}
\label{Eq:AggregateScore}
S(x)
&:=
\int_\Theta s_\theta(x)\,\lambda(d\theta),\\
\label{Eq:TotalScoreDifferenceCost}
d_s(x,y)
&:=
\int_\Theta|s_\theta(x)-s_\theta(y)|\,\lambda(d\theta),\\
\label{Eq:ScoreDecreaseCost}
\ell_s(x,y)
&:=
\int_\Theta(s_\theta(x)-s_\theta(y))_+\,\lambda(d\theta).
\end{align}
The assumptions make these quantities finite and continuous; see
Lemma~\ref{Lem:ScoreRegularity} in Appendix~\ref{App:CSDPreliminaries}.
The characterization uses also the notation 
\begin{align}
\label{Eq:ScoreWassersteinValue}
W_s(P,Q)
&:=
\inf_{\kappa\in\mathcal C(P,Q)}
\int_{D\times D}d_s(x,y)\,\kappa(dx,dy),\\
\label{Eq:AggregateScoreDifference}
\Delta_s(P,Q)
&:=
\int_D S(y)\,Q(dy)-\int_D S(x)\,P(dx).
\end{align}
Thus $W_s(P,Q)$ is the smallest expected total absolute difference between
the selected score values of paired outcomes, while $\Delta_s(P,Q)$ is the
increase in the expected weighted sum $S$ from $P$ to $Q$. Note that only the first
quantity depends on the coupling. Lemma~\ref{Lem:ScoreProfileReduction}
shows that $W_s(P,Q)$ is the usual $1$-Wasserstein distance after each outcome
is mapped to its selected score values.

Define the smallest expected score decrease over all couplings by
\begin{equation}
\label{Eq:MinimumScoreDecrease}
L_s^\star(P,Q)
:=
\inf_{\kappa\in\mathcal C(P,Q)}
\int_{D\times D}\ell_s(x,y)\,\kappa(dx,dy)
\end{equation}
and for $\gamma\in[0,1]$ define
\[
\mathfrak D_{s,\lambda,\gamma}(P,Q)
:=
\inf_{\kappa\in\mathcal C(P,Q)}
\int_{D\times D}
\{\ell_s(x,y)-\gamma\ell_s(y,x)\}\,\kappa(dx,dy).
\]
 The theorem
below proves that the infimum is attained, and hence, 
Definition~\ref{Def:ScoreCompensatedDominance} is equivalent to
$\mathfrak D_{s,\lambda,\gamma}(P,Q)\le0$. Importantly, the theorem also shows that only one
calculation of $W_s(P,Q)$ gives this value for every $\gamma$.

\begin{theorem}[Optimal transport characterization of CSD]
\label{Thm:CompensatedScoreTransport}
For every $0\le\gamma\le1$,
\begin{equation}
\label{Eq:CSDMeanWassersteinIdentity}
\mathfrak D_{s,\lambda,\gamma}(P,Q)
=
(1-\gamma)L_s^\star(P,Q)-\gamma\Delta_s(P,Q)
=
\frac{1-\gamma}{2}W_s(P,Q)
-
\frac{1+\gamma}{2}\Delta_s(P,Q).
\end{equation}
The infima are attained and for $\gamma<1$, the couplings that minimize
$\mathfrak D_{s,\lambda,\gamma}(P,Q)$ are exactly those that attain
\eqref{Eq:ScoreWassersteinValue}.
\end{theorem}

To get some intuition on this result first consider the extreme cases.  
At $\gamma=0$, equation~\eqref{Eq:CSDMeanWassersteinIdentity} gives
$\mathfrak D_{s,\lambda,0}(P,Q)=L_s^\star(P,Q)$. CSD therefore holds exactly
when some coupling has zero expected score decrease. If the score family
implies the coordinatewise order, this is Strassen's monotone coupling and
hence FOSD. At $\gamma=1$,
$\mathfrak D_{s,\lambda,1}(P,Q)=-\Delta_s(P,Q)$, so CSD requires only that
the expected weighted sum of the scores under $Q$ be at least as large as
under $P$.
For $0<\gamma<1$, dominance in CSD depends on the two quantities. The expected weighted score difference provides one measure of the overall improvement under $Q$. Among all couplings of $P$ and $Q$, the best coupling determines the smallest expected score decrease that must be compensated by expected score increase. A smaller $\gamma$ requires more expected score increase to compensate for each unit of expected score decrease. The following corollary formalizes how these two quantities determine dominance in CSD.

\begin{corollary}
\label{Cor:CSDWassersteinCriterion}
For every $\gamma\in[0,1]$,
\begin{equation}
\label{Eq:CSDWassersteinCriterion}
P\preceq_{s,\lambda,\gamma}^{\rm CSD}Q
\quad\Longleftrightarrow\quad
\Delta_s(P,Q)
\ge
\frac{1-\gamma}{1+\gamma}W_s(P,Q).
\end{equation}
\end{corollary}

The same inequality gives the basic order properties.

\begin{proposition}[Order properties]
\label{Prop:CSDOrderStructure}
For fixed $\gamma\in[0,1]$, CSD is reflexive and transitive. If
$P_i\preceq_{s,\lambda,\gamma}^{\rm CSD}Q_i$ for $i=1,2$, then, for every
$\alpha\in[0,1]$,
\[
\alpha P_1+(1-\alpha)P_2
\preceq_{s,\lambda,\gamma}^{\rm CSD}
\alpha Q_1+(1-\alpha)Q_2.
\]
Moreover, $P\preceq_{s,\lambda,\gamma}^{\rm CSD}Q$ implies
$P\preceq_{s,\lambda,\gamma'}^{\rm CSD}Q$ whenever
$0\le\gamma\le\gamma'\le1$. If $\gamma<1$, both
$P\preceq_{s,\lambda,\gamma}^{\rm CSD}Q$ and
$Q\preceq_{s,\lambda,\gamma}^{\rm CSD}P$ hold if and only if
$W_s(P,Q)=0$.
\end{proposition}

Define the smallest tolerance for which CSD dominance holds by
$\gamma_{s,\lambda}^\star(P,Q):=\inf\{\gamma\in[0,1]:
P\preceq_{s,\lambda,\gamma}^{\rm CSD}Q\}$, with
$\inf\varnothing:=+\infty$. The next corollary characterizes the smallest tolerance for which dominance in CSD holds. This value is useful in applications because it gives the minimum amount of compensation required for dominance.

\begin{corollary}[Smallest tolerance]
\label{Cor:CSDCriticalGamma}
The smallest tolerance is
\begin{equation}
\label{Eq:CSDCriticalGamma}
\gamma_{s,\lambda}^\star(P,Q)
=
\begin{cases}
0,
& W_s(P,Q)=0,\\[1mm]
\displaystyle
\frac{W_s(P,Q)-\Delta_s(P,Q)}
     {W_s(P,Q)+\Delta_s(P,Q)},
& W_s(P,Q)>0\text{ and }\Delta_s(P,Q)\ge0,\\[4mm]
+\infty,
& \Delta_s(P,Q)<0.
\end{cases}
\end{equation}
\end{corollary}

When $\Delta_s(P,Q)<0$, even $\gamma=1$ fails because the expected weighted
sum of the scores is lower under $Q$. Hence no tolerance in $[0,1]$ supports the
comparison, which explains the last case in
\eqref{Eq:CSDCriticalGamma}.

The following example shows the compensation calculation when one
distribution is a fixed shift of the other.

\begin{example}[Compensation under a common shift]
\label{example:CSDTranslationShift}
Let $X$ be any integrable random vector with distribution $P$. Fix
$\delta\in\mathbb R^d$ such that $X+\delta\in D$ almost surely, and let
$Q_\delta$ denote the distribution of $X+\delta$. Use the coordinate scores
$s_i(x)=x_i$, $i=1,\ldots,d$, and let $\lambda$ be counting measure on
$\{1,\ldots,d\}$. Appendix~\ref{App:CSDTransportProofs} shows that, for every
$\gamma\in[0,1]$,
\begin{equation}
\label{Eq:CSDTranslationCriterion}
P\preceq_{s,\lambda,\gamma}^{\rm CSD}Q_\delta
\quad\Longleftrightarrow\quad
\sum_{i=1}^d(-\delta_i)_+
\le
\gamma\sum_{i=1}^d(\delta_i)_+.
\end{equation}
 Other score choices can be used to  measure increases and decreases differently. With
nonlinear scores, the change between $x$ and $y$ can depend on their levels,
not only on $y-x$.
\end{example}

\subsection{Computation}
\label{Sec:CSDComputation}

For empirical and other finitely supported distributions, the transport characterization gives a direct way to check whether dominance in CSD holds. We use this formulation in the empirical application in Section~\ref{Sec:OlistApplication}. 

Let $\delta_x$ denote the Dirac probability
measure at $x$, and write $P=\sum_{r=1}^{n_P}p_r\delta_{x_r}$ and
$Q=\sum_{\ell=1}^{n_Q}q_\ell\delta_{y_\ell}$, where all masses are positive
and sum to one. Then $W_s(P,Q)$ is the value of
\begin{equation}
\label{Eq:FiniteScoreWassersteinLP}
\min_{\pi_{r\ell}\ge0}
\left\{
\sum_{r=1}^{n_P}\sum_{\ell=1}^{n_Q}
\pi_{r\ell}d_s(x_r,y_\ell):
\sum_{\ell=1}^{n_Q}\pi_{r\ell}=p_r,\
\sum_{r=1}^{n_P}\pi_{r\ell}=q_\ell
\right\}.
\end{equation}
Solving \eqref{Eq:FiniteScoreWassersteinLP} once gives an optimal coupling for
every $\gamma<1$. After computing $\Delta_s(P,Q)$ from the support points and masses,
\eqref{Eq:CSDMeanWassersteinIdentity}, and
\eqref{Eq:CSDCriticalGamma} give the CSD
value for every $\gamma$, and the smallest tolerance for which the
comparison holds. One solve therefore covers every value of $\gamma$. The linear program has
$n_P n_Q$ nonnegative variables and $n_P+n_Q-1$ independent equality
constraints.

For a finite family of linear scores,
$\lambda=\sum_{j=1}^J\lambda_j\delta_{a^j}$ and
$s_{a^j}(x)=a^j\cdot x$ so 
$d_s(x_r,y_\ell)=\sum_{j=1}^J\lambda_j
|a^j\cdot(x_r-y_\ell)|$. For coordinate scores, it becomes
$d_s(x_r,y_\ell)=\sum_{i=1}^d\beta_i|x_{r,i}-y_{\ell,i}|$. These formulas
construct the transport cost matrix directly from the selected score values.  Additional computational details are given in Appendix
\ref{App:CSDTransportProofs}. 

\subsection{Utility Representation of CSD}
\label{Sec:Utility}

We next identify the utilities whose expected values generate CSD. Define
\begin{equation}
\label{Eq:CSDUtilityClassIncrement}
\mathcal U_{s,\lambda,\gamma}
:=
\left\{
u:\mathbb R^d\to\mathbb R:
 u(x)-u(y)\le \ell_s(x,y)-\gamma\ell_s(y,x)
 \text{ for all }x,y\in\mathbb R^d
\right\}.
\end{equation}
Under our assumptions, every utility in
this class is continuous and 
its expectations under $P$ and $Q$ are well defined. We have the following utility characterization of CSD.

\begin{theorem}[Utility representation]
\label{Thm:CSDUtilityRepresentation}
For every $\gamma\in[0,1]$,
\begin{equation}
\label{Eq:CSDUtilityDuality}
\mathfrak D_{s,\lambda,\gamma}(P,Q)
=
\sup_{u\in\mathcal U_{s,\lambda,\gamma}}
\left\{
\int_D u\,dP-
\int_D u\,dQ
\right\}.
\end{equation}
Consequently, $P\preceq_{s,\lambda,\gamma}^{\rm CSD}Q$ if and only if
$\int_D u\,dP\le\int_D u\,dQ$ for every
$u\in\mathcal U_{s,\lambda,\gamma}$.
\end{theorem}

For $\gamma<1$, the utility class also has the following equivalent form:
$u\in\mathcal U_{s,\lambda,\gamma}$ exactly when
\begin{equation}
\label{Eq:CSDUtilityClassWasserstein}
u(x)
=
c+
\frac{1+\gamma}{2}S(x)
+
\frac{1-\gamma}{2}f(x),
\qquad
|f(x)-f(y)|\le d_s(x,y)
\quad\text{for all }x,y\in\mathbb R^d,
\end{equation}
for some $c\in\mathbb R$ and $f:\mathbb R^d\to\mathbb R$. At $\gamma=1$,
the class reduces to $u(x)=c+S(x)$. Equation~\eqref{Eq:CSDUtilityClassWasserstein} resembles the optimal transport characterization of CSD in Theorem \ref{Thm:CompensatedScoreTransport}. In particular, every admissible
utility is a weighted sum of $S$ plus a second function
whose change between $x$ and $y$ cannot exceed the total score difference
$d_s(x,y)$ while $\gamma$ controls the weight. We next show that this utility characterization recovers known multivariate almost-stochastic-dominance orders defined through utility restrictions in the literature and also yields new, natural utility classes.

\subsubsection{Linear scores}

Suppose $\Theta=\mathbb R_+^d$, $s_a(x)=a\cdot x$, and $\lambda$
is a nonzero finite Borel measure satisfying
$\int_{\mathbb R_+^d}\|a\|_1\,\lambda(da)<\infty$. Define
\[
\mathcal B_{\lambda,\gamma}
:=
\left\{
\int_{\mathbb R_+^d}r(a)a\,\lambda(da):
 r\text{ measurable and }\gamma\le r(a)\le1
\right\}.
\]
The gradient characterization below shows how the selected linear scores restrict marginal utilities and the tradeoffs among them. For example, with two scores $a^1\cdot x$ and $a^2\cdot x$ having weights $\lambda_1$ and $\lambda_2$, every admissible utility satisfies $\nabla u(x)=\lambda_1 r_1(x)a^1+\lambda_2 r_2(x)a^2$ for some $r_1(x),r_2(x)\in[\gamma,1]$. Thus the choice of the score vectors $a^1$ and $a^2$ can restrict not only the marginal utility of each attribute, but also the marginal rates of substitution among attributes.

\begin{proposition}[Linear-score utility order]
\label{Prop:CSDLinearScoresMASD}
Under linear scores, for every $\gamma\in[0,1]$,
$P\preceq_{s,\lambda,\gamma}^{\rm CSD}Q$ if and only if
$\int_D u\,dP\le\int_D u\,dQ$ for every differentiable
$u:\mathbb R^d\to\mathbb R$ satisfying
$\nabla u(x)\in\mathcal B_{\lambda,\gamma}$ for every $x$.
\end{proposition}

\subsubsection{Coordinate scores and substitution-MASD}
Substitution MASD is a special case of the linear-score in which the selected score vectors are the coordinate unit vectors, so each score evaluates a single attribute.
Formally, let $\beta_1,\ldots,\beta_d>0$ and
$\lambda_\beta:=\sum_{i=1}^d\delta_{\beta_i e_i}$. Then
$\mathcal B_{\lambda_\beta,\gamma}
=\prod_{i=1}^d[\gamma\beta_i,\beta_i]$. For $0<\gamma\le1$, define
\[
\mathcal U_{\gamma,\beta}^{\rm MASD}
:=
\left\{
 u:\mathbb R^d\to\mathbb R:
 u\text{ is differentiable and }
 \gamma\beta_i\le\partial_i u(x)\le\beta_i
 \text{ for every }x\text{ and }i
\right\}.
\]
We write $P\preceq_{\gamma,\beta}^{\rm MASD}Q$ when
$\int_D u\,dP\le\int_D u\,dQ$ for every
$u\in\mathcal U_{\gamma,\beta}^{\rm MASD}$. This is the utility class used for substitution-MASD in
\citet{MullerScarsiniTsetlinWinkler2025}, with lower bounds
$\gamma\beta_i$ and upper bounds $\beta_i$.
For coordinate scores, $\Delta_s(P,Q)$ is the weighted difference in coordinate means and $W_s(P,Q)$ is the weighted $\ell_1$-Wasserstein distance. Our results imply the following Wasserstein characterization of substitution-MASD. 

\begin{corollary}
\label{Cor:CSDMASDWasserstein}
For every $0<\gamma\le1$, 
$P\preceq_{s,\lambda_\beta,\gamma}^{\rm CSD}Q
\quad\Longleftrightarrow\quad
P\preceq_{\gamma,\beta}^{\rm MASD}Q$, 
and
\begin{equation}
\label{Eq:MASDWassersteinCharacterization}
P\preceq_{\gamma,\beta}^{\rm MASD}Q
\quad\Longleftrightarrow\quad
\Delta_\beta(P,Q)
\ge
\frac{1-\gamma}{1+\gamma}W_{1,\beta}(P,Q).
\end{equation}
Moreover,
\begin{equation}
\label{Eq:FOSDWeightedWassersteinCharacterization}
Q\text{ dominates }P\text{ by FOSD}
\quad\Longleftrightarrow\quad
\Delta_\beta(P,Q)=W_{1,\beta}(P,Q).
\end{equation}
\end{corollary}

the smallest value of
$\gamma$ for which MASD holds is
given in Corollary~\ref{Cor:CSDCriticalGamma}. In particular, one weighted transport problem determines
whether substitution-MASD holds for every $\gamma$ and gives the smallest value for which it holds.  We now use this characterization to illustrate dominance between two correlated normal random variables. 

\begin{example}
\label{example:MASDGaussianMeanDispersion}

Fix  coordinate weights $\beta=(1,1)$, and let
\[
P
=
\mathcal N\left(
\begin{pmatrix}
0\\
0
\end{pmatrix},
\begin{pmatrix}
1 & \rho\\
\rho & 1
\end{pmatrix}
\right),
\qquad
Q
=
\mathcal N\left(
\begin{pmatrix}
m\\
0
\end{pmatrix},
\begin{pmatrix}
1 & \rho\sigma\\
\rho\sigma & \sigma^2
\end{pmatrix}
\right),
\]
where $m>0$, $\sigma\ge1$, and $\rho\in(-1,1)$. Relative to $P$, the
distribution $Q$ increases the mean of the first coordinate by $m$ and
increases the standard deviation of the second coordinate from $1$ to
$\sigma$, while preserving its mean and the correlation coefficient. Thus the expected sum of the two coordinates is higher under $Q$, while the
second coordinate is more dispersed.

We show in the appendix that for every $\gamma\in(0,1]$,
\begin{equation}
\label{Eq:MASDGaussianMeanDispersionThreshold}
P\preceq_{\gamma,\beta}^{\rm MASD}Q
\quad\Longleftrightarrow\quad
\gamma
\ge
\frac{
(\sigma-1)\sqrt{2/\pi}
}{
2m+(\sigma-1)\sqrt{2/\pi}
}.
\end{equation}
By Corollary \ref{Cor:CSDMASDWasserstein} the same condition characterizes
coordinate-score CSD.

Condition~\eqref{Eq:MASDGaussianMeanDispersionThreshold} characterizes when the increase in the mean of the first coordinate compensates for realizations in which the second coordinate under $Q$ falls below its paired value under $P$. As expected, the smallest required tolerance decreases with $m$, increases with $\sigma$, and does not depend on the common correlation coefficient $\rho$.
\end{example}

\subsubsection{Finitely many scores}

We now generalize  the linear scores to the practical case where we have any finitely 
 many scores (including possibly nonlinear scores). This provides a broad utility characterization for CSD. 
Suppose $\Theta=\{1,\ldots,J\}$ and
$\lambda=\sum_{j=1}^J\lambda_j\delta_j$, where $\lambda_j>0$, and write
$s_1,\ldots,s_J$ for the selected scores.

\begin{proposition}
\label{Prop:CSDFinitelyManyScores}
For every $\gamma\in[0,1]$,
$P\preceq_{s,\lambda,\gamma}^{\rm CSD}Q$ if and only if
\[
\int_D v\bigl(\lambda_1s_1(x),\ldots,\lambda_Js_J(x)\bigr)\,P(dx)
\le
\int_D v\bigl(\lambda_1s_1(y),\ldots,\lambda_Js_J(y)\bigr)\,Q(dy)
\]
for every differentiable $v:\mathbb R^J\to\mathbb R$ satisfying
$\gamma\le\partial_jv(z)\le1$ for every $z\in\mathbb R^J$ and
$j=1,\ldots,J$.
\end{proposition}

Thus, one may view CSD with finite scores as substitute MASD in the score space after the outcome $x$ is mapped to $\bigl(\lambda_1s_1(x),\ldots,\lambda_Js_J(x)\bigr)$. In particular, when the
scores are differentiable and
$u(x)=v\bigl(\lambda_1s_1(x),\ldots,\lambda_Js_J(x)\bigr)$, then
\[
\nabla u(x)
=
\sum_{j=1}^J
\lambda_j
\partial_j v\bigl(\lambda_1s_1(x),\ldots,\lambda_Js_J(x)\bigr)
\nabla s_j(x).
\]
Hence, with finitely many scores, the marginal utility of each outcome coordinate is determined by the scores that depend on that coordinate. These marginal utilities may vary across outcomes when the scores are nonlinear. Moreover, when one score depends on several coordinates, the same score-specific weight enters the marginal utilities of all of those coordinates, capturing substitution among them. This allows modelers to restrict marginal rates of substitution when comparing multivariate distributions, which is useful when partial preferences over tradeoffs between attributes are known or can be elicited. For $s_j(x)=x_j$ and
$\lambda_j=\beta_j$, the characterization reduces to
$\gamma\beta_j\le\partial_j u(x)\le\beta_j$, which gives
$\mathcal U_{\gamma,\beta}^{\rm MASD}$.

\begin{remark}[Preference elicitation through scores]
\label{Rem:CSDScoreElicitation}
\citet{MullerScarsiniTsetlinWinkler2025} show that the transfer
interpretation of substitution-MASD can be used to elicit $\gamma$
without eliciting a full multiattribute utility function. While this is not the focus of the current paper, Proposition
\ref{Prop:CSDFinitelyManyScores} extends this idea to the selected scores.
Instead of asking about tradeoffs between changes in individual attributes,
the analyst can ask about tradeoffs between changes in performance measures
that may combine several attributes.  
\end{remark}

\subsection{Lower-Orthant Scores}
\label{Sec:CSDLowerOrthantScores}

In this section, we consider scores based on lower-orthant indicators, a natural family of nonlinear scores. We show that these scores extend the univariate orders between
FOSD and SOSD \citep{MullerScarsiniTsetlinWinkler2017} to vector outcomes. In the multivariate setting lower-orthant scores generate a family of stochastic orders that interpolates
between multivariate FOSD at $\gamma=0$ and second-degree lower-orthant
dominance at $\gamma=1$ which we now define.  

Let $R$ be a probability law on
$\mathbb R^d$ and write
$F_R(t):=R\{x:x\le t\}$, where inequalities are coordinatewise.
The second-degree lower-orthant dominance 
\citep{OBrienScarsini1991} is defined by 
\begin{equation}
\label{Eq:SecondDegreeLowerOrthant}
\int_{(-\infty,t]}F_Q(z)\,dz
\le
\int_{(-\infty,t]}F_P(z)\,dz
\qquad
\text{for every }t\in\mathbb R^d.
\end{equation}
This order 
coincides with SOSD when $d=1$. 
Let $D=[\underline x,\overline x]\subset\mathbb R^d$ be a nondegenerate
rectangle and assume that the support of $P$ and $Q$ is contained in $D$. For each
$t\ge\underline x$, let $\Theta_t:=[\underline x,t]$ and equip
$\Theta_t$ with Lebesgue measure $\lambda_t$. For $z\in\Theta_t$, define
$s_z(x):=1-\mathbf 1_{\{x\le z\}}$ for $x\in D$ and for $x\notin D$  let $s_z(x):=1-\mathbf 1_{\{\Pi_D(x)\le z\}}$ where  $\Pi_D(x)$ denote its coordinatewise projection onto $D$.

Write $s^{(t)}:=(s_z)_{z\in\Theta_t}$. We say that lower-orthant CSD holds
at tolerance $\gamma$ when
\begin{equation}
\label{Eq:LowerOrthantCSD}
P\preceq_{s^{(t)},\lambda_t,\gamma}^{\rm CSD}Q
\qquad
\text{for every finite }t\ge\underline x.
\end{equation}

Interestingly, \citet{MullerScarsiniTsetlinWinkler2017} define
$(1+\gamma)$-stochastic dominance through increasing utilities satisfying
$0\le\gamma u'(y)\le u'(x)$ whenever $x\le y$. Their integral
characterization shows that the order is equivalent to
\begin{equation}
\label{Eq:CSDOneDimensionalBetweenOrders}
\int_{\underline x}^{t}
\{F_Q(z)-F_P(z)\}_+\,dz
\le
\gamma
\int_{\underline x}^{t}
\{F_P(z)-F_Q(z)\}_+\,dz
\qquad
\text{for every }t\in\mathbb R.
\end{equation}
The next corollary shows that the CSD criterion generated by the lower-orthant scores is precisely
\eqref{Eq:CSDOneDimensionalBetweenOrders}. Thus,
$(1+\gamma)$-stochastic dominance is a one-dimensional special case of
lower-orthant CSD. Appendix~\ref{App:CSDLowerOrthantScores} also derives the utility class generated by these orders, provides the explicit transport cost matrix,  and verifies that our CSD results
apply to these scores.

\begin{corollary}[Lower-orthant CSD]
\label{Cor:CSDLowerOrthant}
At $\gamma=0$, lower-orthant CSD coincides with multivariate FOSD, and at
$\gamma=1$, it coincides with second-degree lower-orthant dominance.
When $d=1$, lower-orthant CSD coincides, for every $\gamma\in[0,1]$, with
the $(1+\gamma)$-stochastic dominance order of
\citet{MullerScarsiniTsetlinWinkler2017}.
\end{corollary}

\section{Reference-Weighted Stochastic Dominance}
\label{Sec:RWSD}

In this section we introduce a relaxation of integral stochastic orders. 
\subsection{Definition and Interpretation}
\label{Sec:RWSDDefinition}

Let $X$ be a Polish space with Borel $\sigma$-algebra $\mathscr X$, and let
$(\Theta,\mathscr A)$ be a measurable space. Write $\mathcal P(X)$ and
$\mathcal P(\Theta)$ for the corresponding sets of probability measures. Let
$\varphi:X\times\Theta\to\mathbb R$ be
$\mathscr X\otimes\mathscr A$-measurable and assume
$\sup_{(x,\theta)\in X\times\Theta}|\varphi(x,\theta)|<\infty$. For
$\theta\in\Theta$, write $\varphi_\theta(x):=\varphi(x,\theta)$ and orient the
test functions so that larger expected values are preferred. For
$P,Q\in\mathcal P(X)$, define
$\Gamma_{P,Q}(\theta):=
\mathbb E_Q[\varphi_\theta]-\mathbb E_P[\varphi_\theta]$.
The family $\Phi=(\varphi_\theta)_{\theta\in\Theta}$ generates the integral
stochastic order in which $Q$ dominates $P$ when
$\Gamma_{P,Q}(\theta)\ge0$ for every $\theta$; see
\citet{Muller1997StochasticOrders}. The measurability and boundedness
assumptions ensure that all expected-value differences and integrated test
functions below are measurable and bounded.

A leading example of an integral stochastic order is FOSD. Suppose $X$ is equipped with a closed partial order
$\le$. A measurable set $A\subseteq X$ is a lower set if $x\in A$ and $y\le x$
imply $y\in A$.

\begin{example}[FOSD]
\label{example:FOSD}
Let $\mathscr L$ be the family of closed lower sets, equipped with a
$\sigma$-algebra $\mathscr A_{\mathscr L}$ for which the  map
$(x,A)\mapsto\mathbf 1_A(x)$ is
$\mathscr X\otimes\mathscr A_{\mathscr L}$-measurable. Set
$\Theta=\mathscr L$ and $\varphi_A(x)=1-\mathbf 1_A(x)$. Then
$\Gamma_{P,Q}(A)=P(A)-Q(A)$. Hence $Q$ dominates $P$ by FOSD exactly when
$P(A)\ge Q(A)$ for every closed lower set $A$; see
\citet{KamaeKrengelOBrien1977}.
\end{example}

Let $m$ and $M$ be probability measures on $(\Theta,\mathscr A)$, and fix
$0\le k<1\le K<\infty$. Assume $km(B)\le KM(B)$ for every
$B\in\mathscr A$. For finite measures $\mu$ and $\nu$, write $\mu\le\nu$
when $\mu(B)\le\nu(B)$ for every measurable $B$. Define the set of admissible probability measures by 
\begin{equation}
\label{Eq:TwoReferenceClass}
\mathcal P_{k,K;m,M}
:=
\left\{
\Pi\in\mathcal P(\Theta):
km\le\Pi\le KM
\right\}
\end{equation}
and
\[
\mathcal U_{\Phi;k,K;m,M}
:=
\left\{
 u(x)=a+\tau\int_\Theta \varphi_\theta(x)\,\Pi(d\theta):
 a\in\mathbb R,\ \tau\ge0,\
 \Pi\in\mathcal P_{k,K;m,M}
\right\}.
\]
An admissible $\Pi$ can be understood as a weighting of the test-functions. The reference measures $m$ and $M$ restrict how this weight may be distributed: for every measurable $B\subseteq\Theta$, $\Pi$ must assign at least $k m(B)$ and at most $K M(B)$ probability to $B$. Thus, $m$ and $M$ specify lower and upper reference measures for the weights, while $k$ and $K$ control how tightly admissible weightings must follow those references. We now define the reference-weighted stochastic dominance.

\begin{definition}[Reference-weighted stochastic dominance (RWSD)]
\label{Def:RWSD}
We write $Q\succeq_{\Phi;k,K;m,M}P$ if
\[
\mathbb E_Q[u]\ge\mathbb E_P[u]
\qquad
\text{for every }u\in\mathcal U_{\Phi;k,K;m,M}.
\]
\end{definition}

To get some intuition on RWSD, in Example~\ref{example:FOSD}, an admissible $\Pi$ is a probability measure on
the family of lower sets and Definition~\ref{Def:RWSD} is equivalent to
\[
\int_{\mathscr L}Q(A)\,\Pi(dA)
\le
\int_{\mathscr L}P(A)\,\Pi(dA)
\qquad
\text{for every }\Pi\in\mathcal P_{k,K;m,M}.
\]
Hence, FOSD requires the inequality above for every probability measure on $\mathscr L$, equivalently for every closed lower set. Each closed lower set represents a lower-tail event, and FOSD permits all weight to be placed on any one such event. Thus, a single unfavorable lower-tail event rules out dominance, which can make FOSD particularly stringent for multivariate outcomes. RWSD keeps the same family of lower sets but restricts the probability measures used to weight them. The reference measures can reflect which lower-tail events matter in the application and how much weight they may receive. Dominance in RWSD holds when the inequality is satisfied for every weighting allowed by these reference restrictions. 

More generally, RWSD leaves the test functions that generate an integral stochastic order unchanged and restricts only the probability measures used to weight them.

\begin{example}
\label{example:RWSDPoverty}
We illustrate how the reference measures can be chosen in an application using
the setting of \citet{DuclosSahnYounger2006}. They use bivariate stochastic
dominance to compare poverty using household expenditure per capita and
children's height-for-age $z$-scores (HAZ), a measure of children's nutritional
and health status. They distinguish union poverty, in which deprivation in
either dimension is sufficient for poverty, from intersection poverty, in which
deprivation in both dimensions is required, and consider different poverty
lines in the two dimensions.
Let $z_1,\ldots,z_R$ and $h_1,\ldots,h_S$ denote expenditure and HAZ poverty
lines of interest. For each pair $(z_r,h_s)$, define
$
A^{I}_{rs}
:=
\{x:x_1\le z_r,\ x_2\le h_s\}
$
and
$
A^{U}_{rs}
:=
\{x:x_1\le z_r\text{ or }x_2\le h_s\}.
$
These lower sets represent intersection and union poverty, respectively, under
the different poverty lines.

Suppose there is no reason ex ante to favor one of these poverty events over
another. A natural choice is to give them equal reference weight:
$
m=M=(2RS)^{-1}
\sum_{r=1}^R\sum_{s=1}^S
(\delta_{A^{I}_{rs}}+\delta_{A^{U}_{rs}}).
$
An admissible measure can then be written as
$
\Pi=
\sum_{r=1}^R\sum_{s=1}^S
\{\pi^{I}_{rs}\delta_{A^{I}_{rs}}
+\pi^{U}_{rs}\delta_{A^{U}_{rs}}\},
$
where the weights sum to one and each lies between
$k/(2RS)$ and $K/(2RS)$. Hence, $k$ and $K$ determine how much more or less
weight can be placed on any particular poverty line or poverty definition
relative to its reference weight. Dominance in RWSD holds when
$
\sum_{r=1}^R\sum_{s=1}^S
\{\pi^{I}_{rs}Q(A^{I}_{rs})+\pi^{U}_{rs}Q(A^{U}_{rs})\}
\le
\sum_{r=1}^R\sum_{s=1}^S
\{\pi^{I}_{rs}P(A^{I}_{rs})+\pi^{U}_{rs}P(A^{U}_{rs})\}
$
for every such admissible weighting. This condition is much less stringent than FOSD. 
One could also study other reference measures that permit greater
weight to poverty lines that are more important in the
application.
\end{example}

%\begin{remark}
%\label{Rem:RWSDRepresentation}
%RWSD depends on the indexed test functions and their reference measures, i.e., two families
%can generate the same exact order but give different RWSD. The analyst must therefore specify the indexed family together with
%$m$ and $M$ to apply RWSD. 
%\end{remark}

\subsection{CVaR Characterization}
\label{Sec:RWSDCVaR}

In this section we show that dominance in RWSD reduces to solving a CVaR optimization problem. We first need the following notations. Define
\begin{equation}
\label{Eq:ResidualReference}
\overline R:=\frac{KM-km}{K-k},
\qquad
C:=\frac{K-k}{1-k}.
\end{equation}
Note that $\overline R$ is a probability
measure, $C\ge1$, and $km+(1-k)\overline R$ belongs to
$\mathcal P_{k,K;m,M}$, so the admissible class is nonempty.
 For a probability measure $\rho$, a bounded measurable $L$, and
$C\ge1$, define
\begin{equation}
\label{Eq:CVAR}
\CVaR_C^\rho(L)
:=
\inf_{\eta\in\mathbb R}
\left\{
\eta+C\int_\Theta(L-\eta)_+\,d\rho
\right\}.
\end{equation}
This is the Rockafellar--Uryasev scalar representation of CVaR at confidence
level $1-1/C$; see Theorem~10 and Equation~(28) in
\citet{RockafellarUryasev2002}. 

%Appendix~\ref{App:RWSDCVaR} gives the equivalent maximization over
%probability measures used in the proof.

\begin{theorem}[CVaR characterization]
\label{Thm:CVAR}
For any $P,Q\in\mathcal P(X)$,
$Q\succeq_{\Phi;k,K;m,M}P$ if and only if
\begin{equation}
\label{Eq:TwoReferenceDominanceCVAR}
-k\int_\Theta \Gamma_{P,Q}\,dm
+
(1-k)\CVaR_C^{\overline R}(-\Gamma_{P,Q})
\le0.
\end{equation}
\end{theorem}

The theorem converts optimization over the admissible probability measures
into a scalar CVaR calculation which is important for computation as we discuss in the next section. 

\subsection{Computation}
\label{Sec:RWSDComputation}

Using the scalar representation in \eqref{Eq:CVAR}, the RWSD condition in
\eqref{Eq:TwoReferenceDominanceCVAR} is equivalent to
\[
\inf_{\eta\in\mathbb R}
\left\{
-k\int_\Theta \Gamma_{P,Q}\,dm
+(1-k)\eta
+(K-k)\int_\Theta(-\Gamma_{P,Q}-\eta)_+\,d\overline R
\right\}
\le0.
\]
Once $\Gamma_{P,Q}$ and its integrals under the reference measures are
computed, checking RWSD requires only the one-dimensional minimization over
$\eta$. 

For finite-support reference measures, the calculation is explicit. Let
$
m:=\sum_{j=1}^J p_j\delta_{\theta_j}$, and 
$M:=\sum_{j=1}^J q_j\delta_{\theta_j}$,
where $\theta_1,\ldots,\theta_J$ are the union of the supports of $m,M$. Feasibility
means $kp_j\le Kq_j$ for every $j$. By
\eqref{Eq:TwoReferenceDominanceCVAR}, RWSD is equivalent to
\begin{equation}
\label{Eq:FiniteWorstLoss}
-k\sum_{j=1}^J p_j\Gamma_{P,Q}(\theta_j)
+
\max_{\substack{0\le v_j\le Kq_j-kp_j\\
                 \sum_{j=1}^Jv_j=1-k}}
\sum_{j=1}^Jv_j\{-\Gamma_{P,Q}(\theta_j)\}
\le0.
\end{equation}
The maximization in \eqref{Eq:FiniteWorstLoss} is a well known  continuous
knapsack problem. It is solved by ordering the losses
$-\Gamma_{P,Q}(\theta_j)$ from largest to smallest and allocating the
remaining probability $1-k$ in that order, up to the bound
$Kq_j-kp_j$ for each $j$.

Computing the $J$ expected-value differences from $N$ observations requires
$JN$ test-function evaluations. We then sort these $J$ differences, which
requires $O(J\log J)$ operations. Thus, treating each test-function evaluation
as one operation, the total computational cost is $O(JN+J\log J)$. Importantly, the outcome dimension can affect the
cost of evaluating each test function and the number of test functions selected
by the analyst but 
the RWSD calculation does not depend on the outcome vectors themselves.

\subsection{Examples: FOSD, Lower-Orthant, and Increasing-Concave Dominance}
\label{Sec:RWSDExamples}

The following examples discuss RWSD under familiar
stochastic orders.

\begin{example}[FOSD under RWSD]
\label{example:FOSD-RWSD}
Consider Example~\ref{example:FOSD}  first on
$X=[\underline x,\overline x]\subset\mathbb R$. The lower sets are
$A_t=[\underline x,t]$, and suppose $m$ and $M$ have densities $w_m$ and
$w_M$. In this case RWSD has an exact marginal-utility interpretation: its
nonconstant absolutely continuous utilities are precisely the increasing
utilities satisfying
\begin{equation}
\label{Eq:RWSDFOSDMarginalUtility}
k w_m(x)
\le
\frac{u'(x)}{\int_{\underline x}^{\overline x}u'(t)\,dt}
\le
K w_M(x)
\qquad\text{for almost every }x.
\end{equation}
Equation~\eqref{Eq:RWSDFOSDMarginalUtility} requires at least
$k w_m(x)$ and allows at most $K w_M(x)$ of normalized marginal utility at
$x$. These bounds resemble those used in univariate almost stochastic
dominance, see
\citet{LeshnoLevy2002}. The derivation is in
Appendix~\ref{App:RWSDExampleDerivations}.

For $d\ge2$, closed lower sets do not have a single scalar index, so a
measure on lower-sets cannot generally be translated into separate
coordinatewise marginal utilities. RWSD is then specified directly by the weights allowed on the lower-sets as in Example \ref{example:RWSDPoverty}. 
\end{example}

\begin{example}[Lower-orthant RWSD]
\label{example:LowerOrthant-RWSD}
For simplicity consider the case 
$X=[\underline x_1,\overline x_1]\times
[\underline x_2,\overline x_2]$ and we use the disjoint index space
\[
\Theta^{\rm LO}=\Theta_1\sqcup\Theta_2\sqcup\Theta_{12},
\qquad
\Theta_i=[\underline x_i,\overline x_i],
\qquad
\Theta_{12}=X,
\]
with test functions
\[
\varphi_{i,t_i}(x)=1-\mathbf 1_{\{x_i\le t_i\}},
\qquad
\varphi_{12,t}(x)=1-\mathbf 1_{\{x_1\le t_1,\ x_2\le t_2\}}.
\]
Suppose the reference measures have component densities with respect to
Lebesgue measure on each component,
\[
(w_{m,1},w_{m,2},w_{m,12})
\qquad\text{and}\qquad
(w_{M,1},w_{M,2},w_{M,12}).
\]
It is immediate that every generated utility is increasing and submodular. Write
$u_i=\partial u/\partial x_i$ and
$u_{12}=\partial^2u/(\partial x_1\partial x_2)$ where these derivatives exist.
The reference bounds required by RWSD translate directly in this case into restrictions on the utilities
generated by RWSD. For a nonconstant  generated utility with normalized scale $\tau=1$, wherever the derivatives exist we have 
\[
k w_{m,12}(x)
\le -u_{12}(x)
\le K w_{M,12}(x)
\quad\text{almost everywhere}.
\]
Thus, the reference measures restrict the degree of substitutability between attributes in the utilities generated
by RWSD. Appendix~\ref{App:RWSDExampleDerivations} provides the derivation and shows further that the marginal reference densities bound the marginal utilities, while
the joint reference densities bound $-u_{12}$ as the inequality above.
\end{example}

\begin{example}[Increasing-concave dominance under RWSD]
\label{example:SOSD-RWSD}
Let $X=\Theta=[0,\overline x]$ and
$\varphi_t(x)=\min\{x,t\}$. Because $\varphi_0\equiv0$, assume that the
reference measures put no mass at $0$. Every admissible integrating measure
is then supported on $(0,\overline x]$. These test functions produce the usual
increasing-concave order on a bounded interval. Write $u'_+$ and $u'_-$ for
the right and left derivatives. A nonconstant increasing concave utility with
$0<u'_+(0)<\infty$ has a unique probability measure $\Pi_u$ satisfying
\[
u(x)=u(0)+u'_+(0)
\int_{(0,\overline x]}\min\{x,t\}\,\Pi_u(dt),
\qquad
\Pi_u([x,\overline x])=\frac{u'_-(x)}{u'_+(0)}.
\]
The standard representation and its endpoint conventions are stated in
Lemma~\ref{Lem:SlopeMeasure} in Appendix~\ref{App:RWSDExampleDerivations}.
Now suppose that $m$ and $M$ have densities $w_m$ and
$w_M$ on $(0,\overline x)$, may have atoms at $\overline x$, and have no
other singular parts. If $u'_+$ is absolutely continuous on
$(0,\overline x)$, the RWSD restriction in this setting becomes
\[
k w_m(t)
\le
\frac{-u''(t)}{u'_+(0)}
\le
K w_M(t)
\quad\text{for almost every }t,
\]
together with the corresponding bounds on the atom at $\overline x$.
Hence $m$ specifies where curvature must be present, while $M$ limits where
curvature may concentrate.

For multivariate outcomes, let $X\subset\mathbb R^d$ be compact and convex. As in the FOSD case, RWSD can be used in this setting too by selecting a family of increasing concave test functions that are interesting for the specific application under consideration. One  natural choice is $\psi((a,t),x)=\min\{a^\top x,t\}$ for compact sets of nonnegative
weights $a$ and caps $t$. 
\end{example}

\section{Finite-Sample Guarantees}
\label{sec:FiniteSamFORBOTHORDERS}

In this section, we establish finite-sample guarantees for both CSD and RWSD. As outlined in the introduction, the key to deriving these guarantees lies in our optimal transport and CVaR characterizations, which reduce verifying both conditions to checking the non-positivity of a scalar parameter. Consequently, establishing these guarantees reduces to constructing a one-sided concentration bound  around the empirical estimator.

\subsection{Finite-Sample Guarantee for CSD}
\label{Sec:CSDStatisticalCertification}

The finite-sample analysis in this section builds directly on Theorem~\ref{Thm:CompensatedScoreTransport}. The identity $\mathfrak D_{s,\lambda,\gamma}(P,Q)=(1-\gamma)L_s^\star(P,Q)-\gamma\Delta_s(P,Q)$ shows that, for every $\gamma$, it is enough to obtain an upper bound on the minimum expected score decrease $L_s^\star(P,Q)$ and a lower bound on the expected weighted score difference $\Delta_s(P,Q)$ to bound $\mathfrak D_{s,\lambda,\gamma}(P,Q)$. Because the same two bounds apply to every $\gamma$ at once, we can use the sample to find the smallest $\gamma$ for which CSD is guaranteed without testing different values of $\gamma$ separately.

Let $X_1,\ldots,X_{N_P}$ be i.i.d.\ from $P$ and
$Y_1,\ldots,Y_{N_Q}$ be i.i.d.\ from $Q$, with the two samples independent.
Define the empirical distributions 
$
\widehat P_{N_P}:=\frac1{N_P}\sum_{i=1}^{N_P}\delta_{X_i}$ and 
$\widehat Q_{N_Q}:=\frac1{N_Q}\sum_{j=1}^{N_Q}\delta_{Y_j}$ 
and
\[
\widehat\Delta_s
:=
\frac1{N_Q}\sum_{j=1}^{N_Q}S(Y_j)
-
\frac1{N_P}\sum_{i=1}^{N_P}S(X_i),
\qquad
\widehat L_s^\star
:=
L_s^\star(\widehat P_{N_P},\widehat Q_{N_Q}).
\]
Assume that $P$ and $Q$ are supported on known sets
$D_P,D_Q\subseteq D$, respectively.

To obtain a  guarantee on $\mathfrak D_{s,\lambda,\gamma}(P,Q)$ from the sample, we construct an upper
confidence bound on $L_s^\star(P,Q)$ and a lower confidence bound on
$\Delta_s(P,Q)$. The confidence bounds below use standard bounded-difference concentration.
For a sample mean, the relevant bound is the range of the observations; here,
we instead use the largest change in the empirical quantity when one
observation is replaced within $D_P$ or $D_Q$. This gives sharper bounds than
using the diameter of the supports directly. 
In particular, we define the following constants 
\[
\begin{aligned}
c_P^L
&:=
\sup_{x,x'\in D_P,\ y\in D_Q}
|\ell_s(x,y)-\ell_s(x',y)|,
&
c_Q^L
&:=
\sup_{x\in D_P,\ y,y'\in D_Q}
|\ell_s(x,y)-\ell_s(x,y')|,\\
c_P^S
&:=
\sup_{x,x'\in D_P}|S(x)-S(x')|,
&
c_Q^S
&:=
\sup_{y,y'\in D_Q}|S(y)-S(y')|.
\end{aligned}
\]
Assume that $c_P^L,c_Q^L,c_P^S,c_Q^S<\infty$. For $\delta\in(0,1)$, define
\[
\begin{aligned}
\overline L_\delta
&:=
\widehat L_s^\star
+
\sqrt{
\frac{\log(2/\delta)}{2}
\left(
\frac{(c_P^L)^2}{N_P}
+
\frac{(c_Q^L)^2}{N_Q}
\right)
},\\
\underline\Delta_\delta
&:=
\widehat\Delta_s
-
\sqrt{
\frac{\log(2/\delta)}{2}
\left(
\frac{(c_P^S)^2}{N_P}
+
\frac{(c_Q^S)^2}{N_Q}
\right)
}.
\end{aligned}
\]
For $\gamma\in[0,1]$, set
$\overline{\mathfrak D}_\delta(\gamma):=
(1-\gamma)\overline L_\delta-\gamma\underline\Delta_\delta$.

\begin{proposition}
\label{Prop:CSDStatisticalCertificate}
With probability at least $1-\delta$,
\begin{equation}
\label{Eq:CSDUniformGammaBound}
\mathfrak D_{s,\lambda,\gamma}(P,Q)
\le
\overline{\mathfrak D}_\delta(\gamma)
\qquad
\text{for every }\gamma\in[0,1].
\end{equation}
Consequently, for any measurable data-dependent
$\widehat\gamma\in[0,1]$,
\begin{equation}
\label{Eq:CSDFalseCertification}
\mathbb P\left(
\overline{\mathfrak D}_\delta(\widehat\gamma)\le0
\text{ and }
P\not\preceq_{s,\lambda,\widehat\gamma}^{\rm CSD}Q
\right)
\le\delta.
\end{equation}
\end{proposition}

The same event gives an upper confidence bound for the smallest population
tolerance also. Define
$\overline\gamma_\delta:=
\inf\{\gamma\in[0,1]:
\overline{\mathfrak D}_\delta(\gamma)\le0\}$, with
$\inf\varnothing:=+\infty$. Since $\overline L_\delta\ge0$,
\begin{equation}
\label{Eq:CSDCriticalGammaUpperBound}
\overline\gamma_\delta
=
\begin{cases}
0,
& \overline L_\delta=0,\\[1mm]
\displaystyle
\frac{\overline L_\delta}
     {\overline L_\delta+\underline\Delta_\delta},
& \overline L_\delta>0
  \text{ and }\underline\Delta_\delta\ge0,\\[4mm]
+\infty,
& \overline L_\delta>0
  \text{ and }\underline\Delta_\delta<0.
\end{cases}
\end{equation}

\begin{corollary}[Confidence bound for the smallest tolerance]
\label{Cor:CSDCriticalGammaCertificate}
With probability at least $1-\delta$,
$\gamma_{s,\lambda}^\star(P,Q)\le\overline\gamma_\delta$ and, whenever
$\overline\gamma_\delta<\infty$,
$P\preceq_{s,\lambda,\gamma}^{\rm CSD}Q$ for every
$\gamma\in[\overline\gamma_\delta,1]$.
\end{corollary}

Because the inequality in \eqref{Eq:CSDUniformGammaBound} holds simultaneously for every $\gamma$, the same sample can be used to determine the smallest $\gamma$ for which the  domination in CSD is guaranteed.   
The result requires deterministic bounds on how much replacing one observation
can change the empirical minimum score decrease and the sample difference
$\widehat\Delta_s$. Appendix~\ref{App:CSDStatProofs}
 discusses these bounds. 
Importantly, once these bounds are fixed, the confidence margins have
no additional term that grows explicitly with the outcome dimension.

\subsection{Finite-Sample Guarantee for RWSD}
\label{Sec:FiniteSampleGuarantees}

In this section we utilize our CVaR formulation to provide finite sample guarantees for RWSD. Suppose that $P$ and $Q$ are observed through independent
samples. For a
bounded measurable $G:\Theta\to\mathbb R$, define
\[
V(G)
:=
\sup_{\Pi\in\mathcal P_{k,K;m,M}}
\int_\Theta G(\theta)\,\Pi(d\theta).
\]
By Definition~\ref{Def:RWSD},  RWSD holds exactly when $V(-\Gamma_{P,Q})$   is nonpositive. We now provide statistical guarantees for this non-positivity for empirical measures.  We need the following notations. 

Let $Z_1^P,\ldots,Z_{n_P}^P$ be i.i.d.\ from $P$, let
$Z_1^Q,\ldots,Z_{n_Q}^Q$ be i.i.d.\ from $Q$, and assume that the samples are
independent. Write
$
\widehat P_{n_P}:=\frac1{n_P}\sum_{i=1}^{n_P}\delta_{Z_i^P} $  and
$\widehat Q_{n_Q}:=\frac1{n_Q}\sum_{j=1}^{n_Q}\delta_{Z_j^Q}$ for the empirical distributions. 
Then
$
-\Gamma_{\widehat P_{n_P},\widehat Q_{n_Q}}(\theta)
=
\mathbb E_{\widehat P_{n_P}}[\varphi_\theta]
-
\mathbb E_{\widehat Q_{n_Q}}[\varphi_\theta]. 
$

Let $b$ be a uniform range bound:
$
\sup_{x,x'\in X}|\varphi_\theta(x)-\varphi_\theta(x')|\le b$
for every $\theta\in\Theta$.

The test-function family, reference measures, constants $k,K$, and range
bound are fixed independently of the samples used for the guarantee. In principle one could use a data-dependent
procedure to choose these parameters by using sample splitting.

\begin{theorem}[Finite-sample RWSD guarantee]
\label{Thm:RWSD_STATS}
The random variable $V(-\Gamma_{\widehat P_{n_P},\widehat Q_{n_Q}})$ is measurable and for every $\delta\in(0,1)$,
\begin{equation}
\label{Eq:FiniteSampleCertificationError}
\mathbb P\left(
V(-\Gamma_{\widehat P_{n_P},\widehat Q_{n_Q}})
+
\varepsilon_{n_P,n_Q}(\delta)
\le0
\text{ and }
Q\not\succeq_{\Phi;k,K;m,M}P
\right)
\le
\delta,
\end{equation}
where
\begin{equation}
\label{Eq:StatMargin}
\varepsilon_{n_P,n_Q}(\delta)
:=
b
\sqrt{
\frac{
\left(\frac1{n_P}+\frac1{n_Q}\right)
\{\log(1/\delta)+(1-k)\log C\}
}{2}
}.
\end{equation}
\end{theorem}

Note that $\varepsilon_{n_P,n_Q}(\delta)$  decreases with the sample sizes at the
usual square-root rate and increases with the requested confidence level. Importantly, the statistical guarantee does not depend directly on the outcome dimension as long as the range bound $b$ is known and  
the evaluation of $\Gamma$ is feasible. When $k=0$, increasing $K$ approaches exact dominance in the corresponding integral stochastic dominance under the
continuity and full-support conditions we discuss in Appendix~\ref{App:RWSDExactLimit}. As expected, in this case the quantity  $\varepsilon_{n_P,n_Q}(\delta)$ approaches to infinity as $C$ tends to infinity in this case 
so the statistical guarantee above becomes uninformative. 

\begin{example}[Finite-sample FOSD under RWSD]
\label{example:FiniteSampleFOSDRWSD}
Consider the setting of Example~\ref{example:FOSD-RWSD}. For each lower set
$A\in\mathscr L$,
\[
-\Gamma_{\widehat P_{n_P},\widehat Q_{n_Q}}(A)
=
\widehat Q_{n_Q}(A)-\widehat P_{n_P}(A)
=
\frac1{n_Q}\sum_{j=1}^{n_Q}\mathbf 1_{\{Z_j^Q\in A\}}
-
\frac1{n_P}\sum_{i=1}^{n_P}\mathbf 1_{\{Z_i^P\in A\}}.
\]
Since $\varphi_A(x)=1-\mathbf 1_A(x)$ has range length one,
Theorem~\ref{Thm:RWSD_STATS} applies with $b=1$. Hence, if
\[
V(-\Gamma_{\widehat P_{n_P},\widehat Q_{n_Q}})
+
\varepsilon_{n_P,n_Q}(\delta)
\le0,
\]
then the corresponding RWSD dominance holds with confidence at
least $1-\delta$.

The same argument applies to any stochastic order generated by bounded test
functions; only the test functions and their range bound $b$ change.
\end{example}

\section{Empirical Illustration with Olist E-Commerce Data}
\label{Sec:OlistApplication}

We use the Brazilian E-Commerce Public Dataset by Olist to illustrate what CSD
and RWSD reveal when empirical multivariate FOSD fails
\citep{OlistBrazilianEcommerce}. We deliberately construct two groups of sellers
using their fulfillment performance before 2018 so that one group has much
better past fulfillment performance than the other. We then compare their later
service outcomes. Thus, the example is designed to create a setting in which
dominance should be particularly plausible. Even in this setting, empirical
FOSD fails because of a small unfavorable difference in one part of the
distribution. CSD and RWSD show that dominance still holds after allowing a
small relaxation. CSD also diagnoses the failure in the sense that it shows how much favorable movement is needed to
compensate for the unfavorable movement and which service measures account for
the failure. 

The data contain orders, seller assignments, delivery dates,
shipping deadlines, and customer reviews. We discuss it in details in the appendix. 
We classify sellers using fulfillment performance observed before
January 1, 2018, averaging their carrier-handoff speed and the share of handoffs
completed by the shipping deadline. We compare the two groups using service outcomes from delivered and reviewed
orders purchased between January 1 and June 30, 2018. We  reweight the orders so that  differences in product mix and purchase timing are not driving the
comparison (see Appendix \ref{App:OlistApplication} for the precise construction). Let $\widehat P$ denote the weighted empirical
distribution of service outcomes for the bottom-quartile sellers and
$\widehat Q$ the corresponding distribution for the top-quartile sellers.
These distributions contain 5,545 and 6,128 orders, respectively.
Each order has eight outcomes in $[0,1]$, with larger values always preferred.
They measure carrier-transit speed, purchase-to-delivery speed, whether delivery
is on time, whether it is no more than three days late, the severity of any
positive delay, the final review score, whether that review is at least three
stars, and whether it is five stars. Appendix~\ref{App:OlistApplication}
defines these coordinates and the order-selection rules.

\subsection{CSD Comparisons}
We consider three CSD comparisons. One of our three CSD comparisons uses four service scores that group the eight outcomes into four service measures. The particular grouping is only one way to
separate the coordinates but an analyst can use different scores when there are specific performance measures of interest.  For the four service scores we define
\begin{equation}
\label{Eq:OlistFourScores}
s_1(x)=x_1,\qquad
s_2(x)=\frac{x_1+3x_2}{4},\qquad
s_3(x)=\frac{x_3+x_4+x_5}{3},\qquad
s_4(x)=\frac{x_6+x_7+x_8}{3}
\end{equation}
with weights $1/6$, $1/3$, $1/4$, and $1/4$. The first score keeps
carrier transit separate, while the second combines carrier transit with the
customer's total delivery time and places more weight on total delivery time.
The last two scores average the three measures of delivery relative to the
promised date and the three review measures, respectively.
We also compare the eight coordinates directly, using weights $1/4$ for
$x_1,x_2$ and $1/12$ for each remaining coordinate. This is the
substitution-MASD case in Section~\ref{Sec:Utility}.

A third comparison uses the lower-tail score idea from
Section~\ref{Sec:CSDLowerOrthantScores}. We ask whether one or several of the
four service measures in \eqref{Eq:OlistFourScores} are simultaneously low.
For each of five quantile levels and each nonempty subset of the four service
measures, we define a lower-tail event in which every measure in that subset is
below its corresponding threshold. The associated score equals zero on this
event and one otherwise, so larger values remain preferred. This gives
$5(2^4-1)=75$ scores. The thresholds, weights, and full construction are given
in Appendix~\ref{App:OlistApplication}.

For each choice of scores, we solve the finite transport problem in
\eqref{Eq:FiniteScoreWassersteinLP} once. Theorem~\ref{Thm:CompensatedScoreTransport}
then combines this transport value with $\Delta_s(\widehat P,\widehat Q)$ to
compute $\mathfrak D_{s,\lambda,\gamma}(\widehat P,\widehat Q)$ for every
$\gamma$ and the smallest value $\gamma^\star$ for which CSD holds. Thus, one
optimization problem gives the full CSD comparison over all values of $\gamma$.
Empirical FOSD nevertheless fails, even though we deliberately constructed the
two seller groups so that one had much better pre-2018 fulfillment performance.
For the coordinate comparison, Corollary~\ref{Cor:CSDMASDWasserstein} requires
$W_{1,\beta}(\widehat P,\widehat Q)=
\Delta_\beta(\widehat P,\widehat Q)$, but the observed gap is
$W_{1,\beta}(\widehat P,\widehat Q)-\Delta_\beta(\widehat P,\widehat Q)
=0.00004466>0$. Panel~(a) of Figure~\ref{Fig:OlistCSDApplication} shows why:
near the carrier-transit threshold $t=0.936$, $\widehat Q$ places about
$0.00515$ more probability below the threshold than $\widehat P$. This single
lower-tail reversal is enough to rule out FOSD, even though the mean of every
service measure in \eqref{Eq:OlistFourScores} is higher under $\widehat Q$.

\begin{table}[t]
\centering
\small
\caption{CSD results for the two reweighted empirical distributions. The last
column reports the minimum expected score decrease as a percentage of the
minimum expected score increase.}
\label{Tab:OlistCSDResults}
\begin{tabular}{@{}lccc@{}}
\toprule
Scores used in the comparison
& $\Delta_s$
& $L_s^\star$
& $\gamma^\star$ (percent) \\
\midrule
Four service scores
& 0.061484 & 0.00001512 & 0.02458\% \\
Eight coordinates (substitution-MASD)
& 0.061484 & 0.00002233 & 0.03631\% \\
Seventy-five lower-tail indicators
& 0.047662 & 0.00003680 & 0.07716\% \\
\bottomrule
\end{tabular}
\end{table}

Table~\ref{Tab:OlistCSDResults} shows that all three comparisons fail at
$\gamma=0$, but as expected the required relaxation is small. For the four service scores,
$\gamma^\star=0.00024579$. Under an optimal coupling, the minimum expected
score decrease is therefore only $0.0246\%$ of the minimum expected score
increase.  Hence, whereas an FOSD check reports only failure, CSD measures its magnitude. In addition, Panel~(c) of Figure~\ref{Fig:OlistCSDApplication} explains where the
compensation comes from for one optimal coupling. In particular,  the failure of FOSD is driven almost entirely by worse
carrier-transit outcomes for the higher-performing seller group, while the other
service measures provide enough improvement to compensate for this loss under
CSD. 

\begin{figure}[!t]
\centering
\includegraphics[width=0.90\textwidth]{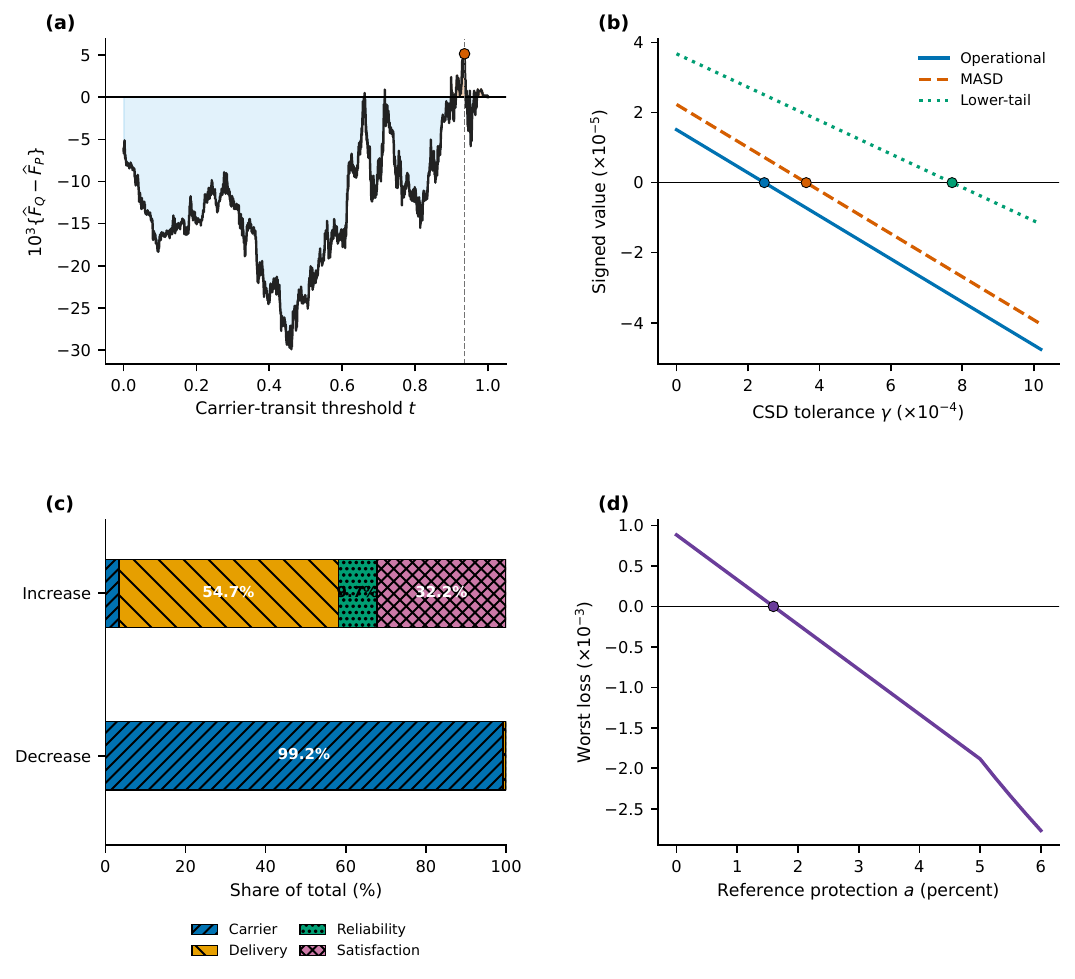}
\caption{Panel (a) shows the carrier-transit threshold at which FOSD
fails. Panel (b) plots the CSD value against $\gamma$ for the three choices of
scores; each marked zero is the smallest value for which CSD holds, computed
from one transport solve. Panel (c) shows which of the four service scores
supply the expected increases and decreases under one optimal coupling. Panel
(d) shows the RWSD calculation below: smaller $a$ allows more weight on an
unfavorable lower-tail event, and the marked zero is the smallest value of
$a$ for which the comparison remains favorable.}
\label{Fig:OlistCSDApplication}
\end{figure}

\subsection{Lower orthant RWSD}

We next apply RWSD to the same 75 lower-tail events used in the third CSD
comparison, following the construction in Example~\ref{example:LowerOrthant-RWSD}. Some of these events are redundant because their pre-2018
thresholds equal one. After removing events that are always satisfied and
combining duplicate events, we obtain 23 distinct lower-tail events
$L_1,\ldots,L_{23}$ with reference weights $w_1,\ldots,w_{23}$.
Appendix~\ref{App:OlistApplication} gives the details.

We set both RWSD reference measures equal to the distribution that assigns
weight $w_j$ to event $L_j$, that is,
$m=M=\sum_{j=1}^{23} w_j\delta_j$.
 For $0<a<1$, take
$k=a$ and $K=1/a$ so the admissible measures satisfy $\sum_{j=1}^{23}\pi_j=1$ and
$aw_j\le\pi_j\le w_j/a$ for $j=1,\ldots,23$. As $a$ decreases, 
more weight can be placed on the worst lower-tail event. Let
\[
V_a
:=
\max_{\substack{\pi_j\ge0,\ \sum_{j=1}^{23}\pi_j=1\\
                   aw_j\le\pi_j\le w_j/a}}
\sum_{j=1}^{23}\pi_j
\{\widehat Q(L_j)-\widehat P(L_j)\},
\qquad
 a^\star:=\inf\{a\in(0,1]:V_a\le0\}.
\]
We can obtain $V_a$ by sorting  as in \eqref{Eq:FiniteWorstLoss} (this is also the CVaR calculation in Theorem~\ref{Thm:CVAR}).

Only one of the 23 lower-tail events is unfavorable: $\widehat Q$ places
$0.000883$ more probability than $\widehat P$ on an event involving carrier
transit. RWSD begins to hold at $a^\star=0.015956$, where the worst allowed
weighting places $98.48\%$ of its mass on this single event. Thus, $\widehat Q$
dominates $\widehat P$ in RWSD unless almost all weight is allowed to
concentrate on this one carrier-transit event.

Note that the values $a^\star$ and $\gamma^\star$ cannot be compared directly because
they measure different relaxations. The value $\gamma^\star$ measures how much
expected score decrease can be compensated by expected score increase within a
coupling, whereas $a^\star$ measures how much weight RWSD can allow on the
most unfavorable lower-tail events before dominance fails.

\section{Conclusions}
\label{Sec:Conclusions}

We develop two relaxations of multivariate stochastic dominance based on its
two standard representations. CSD relaxes the coupling formulation, while RWSD
relaxes the integral formulation of stochastic orders. We characterize the
utility class generated by CSD and in some cases by RWSD, recovering known almost stochastic dominance
orders from the literature as special cases. Importantly, both
relaxations can be computed for empirical distributions using optimal transport
or CVaR and provide finite-sample guarantees that can, in principle, be
computed even when the outcome dimension is high.

We provide an empirical application using Olist data. Even after deliberately
constructing two seller groups so that one has much better past fulfillment
performance, empirical FOSD fails because of a single unfavorable coordinate.
CSD and RWSD show that this failure is small and quantify how much compensation
or concentration on that unfavorable event is needed before the dominance
comparison changes.

We believe these approaches can be useful for comparing general multivariate
distributions when exact stochastic dominance is too restrictive. An interesting
direction for future work is to develop more data-driven methods for choosing
the scores and reference weights. Another direction is to use CSD and RWSD as
constraints or objectives in stochastic optimization problems.

\clearpage
\appendix

\section{Proofs and Technical Details for Compensated Stochastic Dominance}
\subsection{Regularity and transport after mapping outcomes to scores}
\label{App:CSDPreliminaries}

\begin{lemma}[Score regularity]
\label{Lem:ScoreRegularity}
The functions in \eqref{Eq:AggregateScore}--\eqref{Eq:ScoreDecreaseCost}
are finite and continuous. Moreover,
\begin{equation}
\label{Eq:ScoreRegularityBounds}
|S(x)-S(y)|\le d_s(x,y),
\qquad
0\le\ell_s(x,y)\le d_s(x,y)
\le\overline L_s\|x-y\|_1.
\end{equation}
\end{lemma}

\begin{proof}
The assumptions on the weighted score differences give
\[
\begin{aligned}
\int_\Theta|s_\theta(x)|\,\lambda(d\theta)
&\le
\int_\Theta|s_\theta(0)|\,\lambda(d\theta)
+
\int_\Theta|s_\theta(x)-s_\theta(0)|\,\lambda(d\theta)\\
&\le
\int_\Theta|s_\theta(0)|\,\lambda(d\theta)
+\overline L_s\|x\|_1<\infty.
\end{aligned}
\]
Hence $S$, $d_s$, and $\ell_s$ are finite. The last inequality in
\eqref{Eq:ScoreRegularityBounds} is exactly the assumed bound
$d_s(x,y)\le\overline L_s\|x-y\|_1$. The other inequalities in
\eqref{Eq:ScoreRegularityBounds} follow from the triangle inequality and
$0\le t_+\le|t|$. The reverse triangle inequality and the
$1$-Lipschitz property of $t\mapsto t_+$ give
\[
\begin{aligned}
|d_s(x,y)-d_s(x',y')|
&\le d_s(x,x')+d_s(y,y')\\
&\le
\overline L_s\{\|x-x'\|_1+\|y-y'\|_1\},\\
|\ell_s(x,y)-\ell_s(x',y')|
&\le d_s(x,x')+d_s(y,y')\\
&\le
\overline L_s\{\|x-x'\|_1+\|y-y'\|_1\}.
\end{aligned}
\]
The bound for $S$ in \eqref{Eq:ScoreRegularityBounds} proves its continuity,
and the two displayed bounds above prove continuity of $d_s$ and $\ell_s$.
\end{proof}

To compare the laws directly in score space, define
$T_s:\mathbb R^d\to L^1(\lambda)$ by
$T_s(x)(\theta):=s_\theta(x)$. Let $E$ be the closure of
$T_s(\mathbb R^d)$ in $L^1(\lambda)$, set
$\mu:=(T_s)_\#P$ and $\nu:=(T_s)_\#Q$, and write
$(T_s,T_s)(x,y):=(T_s(x),T_s(y))$.

\begin{lemma}[Couplings after mapping outcomes to scores]
\label{Lem:ScoreProfileReduction}
The space $E$ is Polish and
\begin{equation}
\label{Eq:ScoreProfileCouplingReduction}
\mathcal C(\mu,\nu)
=
\left\{
(T_s,T_s)_\#\kappa:
\kappa\in\mathcal C(P,Q)
\right\}.
\end{equation}
Consequently,
\begin{equation}
\label{Eq:ScoreProfileWassersteinIdentity}
W_s(P,Q)=W_1(\mu,\nu),
\end{equation}
where the right-hand side uses the $L^1(\lambda)$ distance on $E$.
In particular, $W_s$ is a pseudometric and satisfies the triangle inequality.
\end{lemma}

\begin{proof} Lemma~\ref{Lem:ScoreRegularity} implies that $T_s:\mathbb R^d\to L^1(\lambda)$ is Lipschitz. Thus $T_s(\mathbb R^d)$ is separable. Its closure $E$ is separable and, as a closed subset of the Banach space $L^1(\lambda)$, complete; hence $E$ is Polish. Moreover, by \eqref{Eq:ScoreRegularityBounds}, \[ \int_E \|z-T_s(0)\|_{L^1(\lambda)}\,\mu(dz) = \int_D d_s(x,0)\,P(dx) \le \overline L_s\int_D\|x\|_1\,P(dx) <\infty, \] and analogously for $\nu$. Thus $\mu$ and $\nu$ have finite first moments on $E$. Apply Lemma~51 in \citet{BauerMemoliNeedhamNishino2025} with $X=Y=\mathbb R^d$, $Z=E$, $f=g=T_s$, $\mu_X=P$, $\mu_Y=Q$, and $p=1$. It gives \eqref{Eq:ScoreProfileCouplingReduction}. Moreover, \[ d_s(x,y) = \|T_s(x)-T_s(y)\|_{L^1(\lambda)}. \] The Wasserstein identity in the same lemma therefore gives \eqref{Eq:ScoreProfileWassersteinIdentity}. The remaining claims follow from the corresponding properties of $W_1$ on $E$.
\end{proof}

\subsection{Optimal transport characterization and order properties}
\label{App:CSDTransportProofs}

\begin{proof}[Proof of Theorem~\ref{Thm:CompensatedScoreTransport}]
By Lemma~\ref{Lem:ScoreRegularity}, $d_s$, $\ell_s(x,y)$, and
$\ell_s(y,x)$ are nonnegative, continuous, and integrable under every
coupling. Regard the couplings as
Borel probability measures on $\mathbb R^d\times\mathbb R^d$. The set
$\mathcal C(P,Q)$ is tight and weakly closed, hence weakly compact by
Prokhorov's theorem, and every element is concentrated on $D\times D$.
The integrals of $d_s(x,y)$, $\ell_s(x,y)$, and $\ell_s(y,x)$ are
weakly lower semicontinuous. Hence the minima in
\eqref{Eq:ScoreWassersteinValue} and \eqref{Eq:MinimumScoreDecrease} are
attained, and the same argument applies to the expected score-increase cost;
see Theorem~4.1 in \citet{Villani2009OptimalTransport}.

Fix $\kappa\in\mathcal C(P,Q)$ and, within this proof, write
\[
L_\kappa:=\int_{D\times D}\ell_s(x,y)\,\kappa(dx,dy),
\qquad
G_\kappa:=\int_{D\times D}\ell_s(y,x)\,\kappa(dx,dy).
\]
The identities
$t_++(-t)_+=|t|$ and $(-t)_+-t_+=-t$, followed by Fubini--Tonelli, give
\begin{equation}
\label{Eq:CSDGainLossDecomposition}
L_\kappa+G_\kappa
=
\int_{D\times D}d_s(x,y)\,\kappa(dx,dy),
\qquad
G_\kappa-L_\kappa
=
\Delta_s(P,Q).
\end{equation}
The second equality depends only on the marginals. Set
$B_\kappa:=L_\kappa+G_\kappa$. Solving
\eqref{Eq:CSDGainLossDecomposition} gives
\[
L_\kappa
=
\frac{B_\kappa-\Delta_s(P,Q)}{2},
\qquad
G_\kappa
=
\frac{B_\kappa+\Delta_s(P,Q)}{2}.
\]
Hence minimizing $B_\kappa$ is equivalent to minimizing either
$L_\kappa$ or $G_\kappa$. Taking minima  shows that the minimum expected
score increase is $\{W_s(P,Q)+\Delta_s(P,Q)\}/2$. Moreover,
\[
L_\kappa-\gamma G_\kappa
=
\frac{1-\gamma}{2}B_\kappa
-
\frac{1+\gamma}{2}\Delta_s(P,Q).
\]
Taking the infimum over $\kappa$ proves
\eqref{Eq:CSDMeanWassersteinIdentity}. If $\gamma<1$, the coefficient of
$B_\kappa$ is positive, so the minimizing couplings are exactly those that attain $W_s(P,Q)$. If $\gamma=1$, the expression equals
$-\Delta_s(P,Q)$ for every coupling. This also proves attainment of the
signed problem.
\end{proof}

\begin{proof}[Proof of Example~\ref{example:CSDTranslationShift}]
For the coordinate scores and counting measure,
\eqref{Eq:AggregateScore} and \eqref{Eq:TotalScoreDifferenceCost} give
\[
S(x)=\sum_{i=1}^d x_i,
\qquad
d_s(x,y)=\sum_{i=1}^d|x_i-y_i|.
\]
The coupling $(X,X+\delta)$ belongs to
$\mathcal C(P,Q_\delta)$ and has expected total coordinate difference
$\sum_{i=1}^d|\delta_i|$. Therefore,
\[
W_s(P,Q_\delta)
\le
\sum_{i=1}^d|\delta_i|.
\]

For any $\kappa\in\mathcal C(P,Q_\delta)$, the triangle inequality and the
fixed marginals imply
\[
\begin{aligned}
\int_{D\times D}d_s(x,y)\,\kappa(dx,dy)
&=
\sum_{i=1}^d
\int_{D\times D}|y_i-x_i|\,\kappa(dx,dy)\\
&\ge
\sum_{i=1}^d
\left|
\int_{D\times D}(y_i-x_i)\,\kappa(dx,dy)
\right|\\
&=
\sum_{i=1}^d|\delta_i|.
\end{aligned}
\]
Taking the infimum over $\kappa$ yields
\[
W_s(P,Q_\delta)
=
\sum_{i=1}^d|\delta_i|.
\]
Moreover, \eqref{Eq:AggregateScoreDifference} gives
\[
\Delta_s(P,Q_\delta)
=
\sum_{i=1}^d\delta_i.
\]

Within this proof, set
\[
A:=\sum_{i=1}^d(\delta_i)_+,
\qquad
B:=\sum_{i=1}^d(-\delta_i)_+.
\]
The preceding identities give
\[
W_s(P,Q_\delta)=A+B,
\qquad
\Delta_s(P,Q_\delta)=A-B.
\]
Using \eqref{Eq:CSDWassersteinCriterion},
\[
P\preceq_{s,\lambda,\gamma}^{\rm CSD}Q_\delta
\quad\Longleftrightarrow\quad
A-B
\ge
\frac{1-\gamma}{1+\gamma}(A+B).
\]
Multiplying by $1+\gamma$ and collecting terms shows that the last inequality
is equivalent to $B\le\gamma A$, which is precisely
\eqref{Eq:CSDTranslationCriterion}.
\end{proof}

\begin{proof}[Proof of Proposition~\ref{Prop:CSDOrderStructure}]
Set $c_\gamma:=(1-\gamma)/(1+\gamma)$. Reflexivity follows from
$W_s(P,P)=\Delta_s(P,P)=0$. Let $R$ be another Borel probability measure
supported on $D$ with a finite first moment. Equation
\eqref{Eq:ScoreProfileWassersteinIdentity} and the triangle inequality for
$W_1$ give
\[
W_s(P,R)\le W_s(P,Q)+W_s(Q,R).
\]
If $P\preceq_{s,\lambda,\gamma}^{\rm CSD}Q$ and
$Q\preceq_{s,\lambda,\gamma}^{\rm CSD}R$, then
\[
\Delta_s(P,R)
=
\Delta_s(P,Q)+\Delta_s(Q,R)
\ge
c_\gamma\{W_s(P,Q)+W_s(Q,R)\}
\ge
c_\gamma W_s(P,R).
\]
Corollary~\ref{Cor:CSDWassersteinCriterion} gives transitivity.

For $\alpha\in[0,1]$, mixing couplings gives
\[
W_s\{\alpha P_1+(1-\alpha)P_2,
      \alpha Q_1+(1-\alpha)Q_2\}
\le
\alpha W_s(P_1,Q_1)+(1-\alpha)W_s(P_2,Q_2).
\]
If $P_i\preceq_{s,\lambda,\gamma}^{\rm CSD}Q_i$ for $i=1,2$, affinity of
$\Delta_s$ and \eqref{Eq:CSDWassersteinCriterion} imply
\[
\begin{aligned}
&\Delta_s\{\alpha P_1+(1-\alpha)P_2,
             \alpha Q_1+(1-\alpha)Q_2\}\\
&\quad\ge
c_\gamma\{\alpha W_s(P_1,Q_1)+(1-\alpha)W_s(P_2,Q_2)\}\\
&\quad\ge
c_\gamma W_s\{\alpha P_1+(1-\alpha)P_2,
                 \alpha Q_1+(1-\alpha)Q_2\}.
\end{aligned}
\]
Thus the mixture property in Proposition~\ref{Prop:CSDOrderStructure} holds. Since $c_\gamma$ is nonincreasing,
\eqref{Eq:CSDWassersteinCriterion} also gives nesting in $\gamma$.

Suppose $\gamma<1$. Applying
\eqref{Eq:CSDWassersteinCriterion} in both directions gives
\[
\Delta_s(P,Q)\ge c_\gamma W_s(P,Q),
\qquad
-\Delta_s(P,Q)\ge c_\gamma W_s(P,Q),
\]
so $W_s(P,Q)=0$. Conversely, $|\Delta_s(P,Q)|\le W_s(P,Q)$, so $W_s(P,Q)=0$ gives
$\Delta_s(P,Q)=0$, and
\eqref{Eq:CSDWassersteinCriterion} then holds in both directions.
\end{proof}

\begin{proof}[Proof of Corollary~\ref{Cor:CSDCriticalGamma}]
Solving \eqref{Eq:CSDWassersteinCriterion} for $\gamma$ gives
\eqref{Eq:CSDCriticalGamma}. Now
suppose $W_s(P,Q)>0$ and $\Delta_s(P,Q)\ge0$. For
$\kappa\in\mathcal C(P,Q)$, define
\[
L_\kappa:=\int_{D\times D}\ell_s(x,y)\,\kappa(dx,dy),
\qquad
G_\kappa:=\int_{D\times D}\ell_s(y,x)\,\kappa(dx,dy),
\qquad
B_\kappa:=L_\kappa+G_\kappa.
\]
Then
$B_\kappa\ge W_s(P,Q)>0$, so $G_\kappa>0$. By
\eqref{Eq:CSDGainLossDecomposition},
\[
\frac{L_\kappa}{G_\kappa}
=
\frac{B_\kappa-\Delta_s(P,Q)}
     {B_\kappa+\Delta_s(P,Q)}.
\]
For $B>0$, the derivative of
$\{B-\Delta_s(P,Q)\}/\{B+\Delta_s(P,Q)\}$ is
$2\Delta_s(P,Q)/\{B+\Delta_s(P,Q)\}^2\ge0$. Hence every coupling attaining
$W_s(P,Q)$ minimizes the ratio, and its value is the finite expression in
\eqref{Eq:CSDCriticalGamma}.
\end{proof}

\paragraph{Cost construction for finite linear score families.}
If $\lambda=\sum_{j=1}^J\lambda_j\delta_{a^j}$ and
$s_{a^j}(x)=a^j\cdot x$, precomputing the score values and constructing the
cost matrix requires
$O\{Jd(n_P+n_Q)+J n_P n_Q\}$ arithmetic operations. For coordinate scores,
constructing the cost matrix requires $O(n_P n_Q d)$. For a fixed $\gamma$, the
signed coordinate cost is
$\sum_{i=1}^d\beta_i\{(x_i-y_i)_+-\gamma(y_i-x_i)_+\}$.

\subsection{Utility representation, linear scores, and finitely many scores}
\label{App:CSDUtilityProofs}
For the rest of the appendix we define $$g_{s,\lambda,\gamma}(x,y):=\ell_s(x,y)-\gamma\ell_s(y,x).$$
\begin{proof}[Proof of Theorem~\ref{Thm:CSDUtilityRepresentation}]
The identity
$t_+-\gamma(-t)_+=\{(1+\gamma)t+(1-\gamma)|t|\}/2$ gives
\begin{equation}
\label{Eq:CSDCostDecomposition}
g_{s,\lambda,\gamma}(x,y)
=
\frac{1+\gamma}{2}\{S(x)-S(y)\}
+
\frac{1-\gamma}{2}d_s(x,y).
\end{equation}
Let
\[
\mathcal F_s
:=
\left\{
f:\mathbb R^d\to\mathbb R:
|f(x)-f(y)|\le d_s(x,y)
\text{ for all }x,y
\right\}.
\]
For $\gamma<1$, applying the increment restriction in both directions and
using \eqref{Eq:CSDCostDecomposition} shows that
$u\in\mathcal U_{s,\lambda,\gamma}$ if and only if
\eqref{Eq:CSDUtilityClassWasserstein} holds for some $c\in\mathbb R$ and
$f\in\mathcal F_s$. For $\gamma=1$,
\eqref{Eq:CSDCostDecomposition} and the increment restriction applied in
both directions give $u=c+S$.

Let $E$, $\mu$, and $\nu$ be as in
Lemma~\ref{Lem:ScoreProfileReduction}. If $f\in\mathcal F_s$ and
$T_s(x)=T_s(y)$, then $f(x)=f(y)$. Hence $f$ induces a $1$-Lipschitz function
on $T_s(\mathbb R^d)$ and a unique $1$-Lipschitz extension to $E$.
Conversely, every $1$-Lipschitz function on $E$ pulls back through $T_s$ to an
element of $\mathcal F_s$. Theorem~5.10 and Remark~6.5 in
\citet{Villani2009OptimalTransport}, together with
\eqref{Eq:ScoreProfileWassersteinIdentity}, give
\[
W_s(P,Q)
=
\sup_{f\in\mathcal F_s}
\left\{
\int_D f\,dP-
\int_D f\,dQ
\right\}.
\]

For $\gamma<1$, substituting \eqref{Eq:CSDUtilityClassWasserstein} into the
right-hand side of \eqref{Eq:CSDUtilityDuality} gives
\[
\frac{1-\gamma}{2}W_s(P,Q)
-
\frac{1+\gamma}{2}\Delta_s(P,Q),
\]
which equals $\mathfrak D_{s,\lambda,\gamma}(P,Q)$ by
\eqref{Eq:CSDMeanWassersteinIdentity}. For $\gamma=1$, the equality follows
from $\mathcal U_{s,\lambda,1}=\{c+S:c\in\mathbb R\}$. The order equivalence
then follows because CSD holds exactly when the signed transport value is
nonpositive.
\end{proof}

\begin{proof}[Proof of Proposition~\ref{Prop:CSDLinearScoresMASD}]
The order interval
$\{r\in L^\infty(\lambda):\gamma\le r\le1\}$ is weak-* compact and convex.
The map $r\mapsto\int r(a)a\,\lambda(da)$ is weak-* continuous, so
$\mathcal B_{\lambda,\gamma}$ is compact and convex. Its support function
satisfies, for every $v\in\mathbb R^d$,
\[
\begin{aligned}
\sigma_{\mathcal B_{\lambda,\gamma}}(v)
&:=
\sup_{b\in\mathcal B_{\lambda,\gamma}}b\cdot v\\
&=
\int_{\mathbb R_+^d}
\left\{
(a\cdot v)_+-\gamma(-a\cdot v)_+
\right\}\lambda(da).
\end{aligned}
\]
The pointwise maximizer takes $r(a)=1$ when $a\cdot v\ge0$ and
$r(a)=\gamma$ otherwise. Consequently,
\begin{equation}
\label{Eq:LinearScoreCostSupportFunction}
g_{s,\lambda,\gamma}(x,y)
=
\sigma_{\mathcal B_{\lambda,\gamma}}(x-y).
\end{equation}

If $u$ is differentiable and $\nabla u(x)\in\mathcal B_{\lambda,\gamma}$
for every $x$, apply the one-dimensional mean-value theorem to
$h(t):=u(y+t(x-y))$ on $[0,1]$. For some $t_0\in(0,1)$,
\[
u(x)-u(y)
=\nabla u(y+t_0(x-y))\cdot(x-y)
\le\sigma_{\mathcal B_{\lambda,\gamma}}(x-y).
\]
Thus $u\in\mathcal U_{s,\lambda,\gamma}$.

Conversely, let $u\in\mathcal U_{s,\lambda,\gamma}$ and set
$A_\lambda:=\int_{\mathbb R_+^d}\|a\|_1\,\lambda(da)$. From
\eqref{Eq:LinearScoreCostSupportFunction},
\[
\sigma_{\mathcal B_{\lambda,\gamma}}(v)
\le\int|a\cdot v|\,\lambda(da)
\le A_\lambda\|v\|_\infty.
\]
Applying the increment restriction in both directions gives
$|u(x)-u(y)|\le A_\lambda\|x-y\|_\infty$, so $u$ is globally Lipschitz.
Let $u_\varepsilon$ be a standard mollification.
Translation invariance of the right-hand side of
\eqref{Eq:LinearScoreCostSupportFunction} implies
\[
u_\varepsilon(x)-u_\varepsilon(y)
\le
\sigma_{\mathcal B_{\lambda,\gamma}}(x-y).
\]
For $t>0$, positive homogeneity of the support function gives
\[
\frac{u_\varepsilon(x+t v)-u_\varepsilon(x)}{t}
\le
\sigma_{\mathcal B_{\lambda,\gamma}}(v).
\]
Letting $t\downarrow0$ yields
$\nabla u_\varepsilon(x)\cdot v
\le\sigma_{\mathcal B_{\lambda,\gamma}}(v)$ for every $v$. The
support-function characterization of a nonempty closed convex set therefore
yields
$\nabla u_\varepsilon(x)\in\mathcal B_{\lambda,\gamma}$ for every $x$.
Thus $u_\varepsilon$ satisfies the differentiability and gradient conditions
in Proposition~\ref{Prop:CSDLinearScoresMASD}.

Let $L_u$ be a global Lipschitz constant for $u$, and write the mollification
as
$u_\varepsilon(x)=\int_{\mathbb R^d}u(x-\varepsilon z)\rho(z)\,dz$
for a compactly supported standard mollifier $\rho$. Then
\[
\|u_\varepsilon-u\|_\infty
\le
L_u\varepsilon
\int_{\mathbb R^d}\|z\|_1\rho(z)\,dz
\longrightarrow0.
\]
Thus, if the expected-utility inequality holds for every differentiable
utility satisfying the gradient condition in the proposition, it also holds
for any $u\in\mathcal U_{s,\lambda,\gamma}$ by applying it to
$u_\varepsilon$ and passing to the limit. The reverse implication follows from
the first part of the proof, which shows that every such differentiable
utility belongs to $\mathcal U_{s,\lambda,\gamma}$. Hence the two utility
conditions induce the same stochastic order, and
Theorem~\ref{Thm:CSDUtilityRepresentation} completes the proof.
\end{proof}

\begin{proof}[Proof of Corollary~\ref{Cor:CSDMASDWasserstein}]
Direct calculation gives
\[
\mathcal B_{\lambda_\beta,\gamma}
=\prod_{i=1}^d[\gamma\beta_i,\beta_i],
\qquad
S(x)=\sum_i\beta_i x_i,
\qquad
d_s(x,y)=\sum_i\beta_i|x_i-y_i|.
\]
Proposition~\ref{Prop:CSDLinearScoresMASD} and
Corollary~\ref{Cor:CSDWassersteinCriterion} therefore give
\eqref{Eq:MASDWassersteinCharacterization}.

It remains to prove \eqref{Eq:FOSDWeightedWassersteinCharacterization}. If
$Q$ dominates $P$ by FOSD, Strassen's theorem provides a coupling with $x\le y$
almost surely. Under that coupling,
$\sum_i\beta_i|x_i-y_i|=\sum_i\beta_i(y_i-x_i)$, so
$W_{1,\beta}(P,Q)\le\Delta_\beta(P,Q)$. For every coupling, the
expected absolute weighted difference is at least the absolute weighted mean
difference, which gives the reverse inequality.

Conversely, suppose equality holds and take a coupling attaining
$W_{1,\beta}(P,Q)$. Its expected weighted
score-decrease term is zero. Since every $\beta_i$ is positive,
$(x_i-y_i)_+=0$ almost surely for every $i$. Thus $x\le y$ almost surely,
and Strassen's theorem gives FOSD; see \citet{Strassen1965}.
\end{proof}

\begin{proof}[Proof of Example~\ref{example:MASDGaussianMeanDispersion}]
For $\beta=(1,1)$, we have 
$
\Delta_\beta(P,Q)=m.
$ 

We next compute $W_{1,\beta}(P,Q)$. Fix
$\kappa\in\mathcal C(P,Q)$. For the first coordinate,
\[
\begin{aligned}
\int_{\mathbb R^2\times\mathbb R^2}
|x_1-y_1|\,\kappa(dx,dy)
&\ge
\left|
\int_{\mathbb R^2}y_1\,Q(dy)
-
\int_{\mathbb R^2}x_1\,P(dx)
\right| = m.
\end{aligned}
\]
For the second coordinate, the reverse triangle inequality gives
\[
\begin{aligned}
\int_{\mathbb R^2\times\mathbb R^2}
|x_2-y_2|\,\kappa(dx,dy)
&\ge
\int_{\mathbb R^2\times\mathbb R^2}
\bigl||y_2|-|x_2|\bigr|\,\kappa(dx,dy)\\
&\ge
\left|
\int_{\mathbb R^2}|y_2|\,Q(dy)
-
\int_{\mathbb R^2}|x_2|\,P(dx)
\right|\\
&=
(\sigma-1)\sqrt{\frac{2}{\pi}},
\end{aligned}
\]
because the second marginals of $P$ and $Q$ are
$\mathcal N(0,1)$ and $\mathcal N(0,\sigma^2)$, respectively. Therefore,
\[
W_{1,\beta}(P,Q)
\ge
m+(\sigma-1)\sqrt{\frac{2}{\pi}}.
\]

To prove the reverse inequality, let
\[
(Z_1,Z_2)
\sim
\mathcal N\left(
\begin{pmatrix}
0\\
0
\end{pmatrix},
\begin{pmatrix}
1 & \rho\\
\rho & 1
\end{pmatrix}
\right)
\]
and define
\[
X=(Z_1,Z_2),
\qquad
Y=(m+Z_1,\sigma Z_2).
\]
Then $X\sim P$ and $Y\sim Q$. The expected cost of this coupling is
\[
\begin{aligned}
\mathbb E\left[
|X_1-Y_1|+|X_2-Y_2|
\right]
&=
m+(\sigma-1)\mathbb E[|Z_2|]\\
&=
m+(\sigma-1)\sqrt{\frac{2}{\pi}}.
\end{aligned}
\]
Hence
\[
W_{1,\beta}(P,Q)
=
m+(\sigma-1)\sqrt{\frac{2}{\pi}}.
\]

By \eqref{Eq:MASDWassersteinCharacterization},
$P\preceq_{\gamma,\beta}^{\rm MASD}Q$ if and only if
\[
m
\ge
\frac{1-\gamma}{1+\gamma}
\left\{
m+(\sigma-1)\sqrt{\frac{2}{\pi}}
\right\}.
\]
This inequality is equivalent to
\[
\gamma
\left\{
2m+(\sigma-1)\sqrt{\frac{2}{\pi}}
\right\}
\ge
(\sigma-1)\sqrt{\frac{2}{\pi}},
\]
which proves
\eqref{Eq:MASDGaussianMeanDispersionThreshold}. 
\end{proof}

\begin{proof}[Proof of Proposition~\ref{Prop:CSDFinitelyManyScores}]
Within this proof, let
$T_\lambda(x):=(\lambda_1s_1(x),\ldots,\lambda_Js_J(x))$, and set
$\widetilde P:=(T_\lambda)_\#P$ and
$\widetilde Q:=(T_\lambda)_\#Q$.
For $x,y\in D$, the definition of $T_\lambda$ gives
\[
\sum_{j=1}^J
\{\lambda_j s_j(x)-\lambda_j s_j(y)\}_+
=
\sum_{j=1}^J
\lambda_j\{s_j(x)-s_j(y)\}_+
=
\ell_s(x,y).
\]
The same identity with $x$ and $y$ exchanged gives the score-increase cost.
Applying \eqref{Eq:ScoreProfileCouplingReduction} to the scaled scores
$\lambda_j s_j$, $j=1,\ldots,J$, with counting measure shows that
CSD holds if and only if some
$\pi\in\mathcal C(\widetilde P,\widetilde Q)$ satisfies
\[
\int_{\mathbb R^J\times\mathbb R^J}
\sum_{j=1}^J(z_j-w_j)_+\,\pi(dz,dw)
\le
\gamma
\int_{\mathbb R^J\times\mathbb R^J}
\sum_{j=1}^J(w_j-z_j)_+\,\pi(dz,dw).
\]
This is CSD for the coordinate scores on $\mathbb R^J$ with unit weights.
For these scores, the admissible gradient set in
Proposition~\ref{Prop:CSDLinearScoresMASD} is $[\gamma,1]^J$. The
proposition therefore implies that this relation holds if and only if
\[
\int_{\mathbb R^J}v(z)\,\widetilde P(dz)
\le
\int_{\mathbb R^J}v(z)\,\widetilde Q(dz)
\qquad
\text{for every differentiable }v\text{ with }
\gamma\le\partial_jv\le1\text{ for all }j.
\]
Substituting the definitions of $\widetilde P$ and $\widetilde Q$ gives
the expected-utility inequalities in
Proposition~\ref{Prop:CSDFinitelyManyScores}.
\end{proof}

\subsection{Lower-orthant scores}
\label{App:CSDLowerOrthantScores}

\noindent\textbf{Explicit formulas for the aggregate score and transport cost.}
Fix a finite cutoff $t\ge\underline x$. Write $S_t$, $d_t$, and $\ell_t$
for the quantities in \eqref{Eq:AggregateScore},
\eqref{Eq:TotalScoreDifferenceCost}, and \eqref{Eq:ScoreDecreaseCost},
respectively, for the score family $s^{(t)}$. Let
$x\vee y$ denote the coordinatewise maximum and let
$\operatorname{vol}$ denote Lebesgue volume.
For $x\in D$, define
$ 
R_t(x):=\{z\in\Theta_t:x\le z\}$ and 
$r_t(x):=\prod_{i=1}^d(t_i-x_i)_+.
$ 
The set $R_t(x)$ is $[x,t]$ when $x\le t$ and is empty otherwise.
Hence $\lambda_t(R_t(x))=r_t(x)$. Since
$s_z(x)=1-\mathbf 1_{\{z\in R_t(x)\}}$,
\begin{equation}
\label{Eq:LowerOrthantAggregateScore}
S_t(x)
=
\operatorname{vol}(\Theta_t)
-
\prod_{i=1}^d(t_i-x_i)_+.
\end{equation}

Moreover,
$|s_z(x)-s_z(y)|
=\mathbf 1_{\{z\in R_t(x)\mathbin{\triangle}R_t(y)\}}$ and
$R_t(x)\cap R_t(y)=R_t(x\vee y)$. Therefore,
\begin{equation}
\label{Eq:LowerOrthantProfileCost}
d_t(x,y)
=
\prod_{i=1}^d(t_i-x_i)_+
+
\prod_{i=1}^d(t_i-y_i)_+
-
2\prod_{i=1}^d\{t_i-(x_i\vee y_i)\}_+.
\end{equation}

Similarly,
$(s_z(x)-s_z(y))_+
=\mathbf 1_{\{z\in R_t(y)\setminus R_t(x)\}}$. Thus,
\begin{equation}
\label{Eq:LowerOrthantScoreDecreaseCost}
\ell_t(x,y)
=
S_t(x\vee y)-S_t(y).
\end{equation}

\noindent\textbf{Utility class for a fixed cutoff.}
The verification below shows that the assumptions for
Theorem~\ref{Thm:CSDUtilityRepresentation} hold for $s^{(t)}$.
By \eqref{Eq:LowerOrthantScoreDecreaseCost},
\[
\ell_t(x,y)-\gamma\ell_t(y,x)
=
\gamma S_t(x)
+
(1-\gamma)S_t(x\vee y)
-
S_t(y).
\]
Hence Theorem~\ref{Thm:CSDUtilityRepresentation} and
\eqref{Eq:CSDUtilityClassIncrement} imply that
$P\preceq_{s^{(t)},\lambda_t,\gamma}^{\rm CSD}Q$ if and only if
$\mathbb E_P[u]\le\mathbb E_Q[u]$ for every $u:D\to\mathbb R$ satisfying\footnote{Note that we restrict the utility class to $D$ while the paper's result consider $\mathbb{R}^{d}$. To justify this, let $\Pi_D$ denote the
coordinatewise projection onto $D$. If $u:D\to\mathbb R$ satisfies
\eqref{Eq:LowerOrthantUtilityRestriction}, define
$\widetilde u(x):=u(\Pi_D(x))$ for $x\in\mathbb R^d$.
For the projected scores, the right-hand side of
\eqref{Eq:CSDUtilityClassIncrement} at $(x,y)$ is the right-hand side of
\eqref{Eq:LowerOrthantUtilityRestriction} evaluated at
$(\Pi_D(x),\Pi_D(y))$. Thus $\widetilde u$ satisfies
\eqref{Eq:CSDUtilityClassIncrement}. Conversely, restricting any utility
in $\mathcal U_{s^{(t)},\lambda_t,\gamma}$ to $D$ gives
\eqref{Eq:LowerOrthantUtilityRestriction}.}
\begin{equation}
\label{Eq:LowerOrthantUtilityRestriction}
u(x)-u(y)
\le
\gamma S_t(x)
+
(1-\gamma)S_t(x\vee y)
-
S_t(y)
\qquad
\text{for all }x,y\in D.
\end{equation}

For $x\le y$, applying
\eqref{Eq:LowerOrthantUtilityRestriction} to $(x,y)$ and $(y,x)$ gives
\[
\gamma\{S_t(y)-S_t(x)\}
\le
u(y)-u(x)
\le
S_t(y)-S_t(x).
\]
At $\gamma=1$, applying
\eqref{Eq:LowerOrthantUtilityRestriction} in both directions for arbitrary
$x,y$ gives $u(x)-u(y)=S_t(x)-S_t(y)$. Hence $u=c+S_t$ on $D$ for some
$c\in\mathbb R$.

\noindent\textbf{Verification of the CSD assumptions.}
We verify the assumptions stated before
Definition~\ref{Def:ScoreCompensatedDominance}.
The map $(z,x)\mapsto s_z(x)$ is jointly measurable. Each score is
coordinatewise nondecreasing because $\Pi_D$ is order preserving.
For $x,y\in D$, $d_t(x,y)$ is the volume of
$R_t(x)\mathbin{\triangle}R_t(y)$. Change the coordinates of $x$ to those
of $y$ one at a time. If two successive points differ only in coordinate
$i$, the symmetric difference of their corresponding sets is contained in
a slab of width $|x_i-y_i|$ and cross-sectional volume at most
$\prod_{j\ne i}(t_j-\underline x_j)$. Therefore,
\[
d_t(x,y)
\le
\sum_{i=1}^d
|x_i-y_i|
\prod_{j\ne i}(t_j-\underline x_j).
\]
Setting
$
\overline L_t
:=
\max_{1\le i\le d}
\prod_{j\ne i}(t_j-\underline x_j),
$ 
we obtain $d_t(x,y)\le\overline L_t\|x-y\|_1$ for all $x,y\in D$.
For the extension to $\mathbb R^d$ used in the main text,
\[
\int_{\Theta_t}|s_z(x)-s_z(y)|\,dz
=
d_t\bigl(\Pi_D(x),\Pi_D(y)\bigr).
\]
Since $\Pi_D$ is $1$-Lipschitz under $\|\cdot\|_1$, it follows that
\[
\int_{\Theta_t}|s_z(x)-s_z(y)|\,dz
\le
\overline L_t\|x-y\|_1
\qquad
\text{for all }x,y\in\mathbb R^d.
\]
Finally, $\lambda_t(\Theta_t)<\infty$ and every score takes values in
$\{0,1\}$, so
$\int_{\Theta_t}|s_z(0)|\,dz\le\lambda_t(\Theta_t)<\infty$.
Thus the assumptions hold for $s^{(t)}$. Consequently,
Lemma~\ref{Lem:ScoreRegularity},
Lemma~\ref{Lem:ScoreProfileReduction},
Theorem~\ref{Thm:CSDUtilityRepresentation}, and
Corollary~\ref{Cor:CSDWassersteinCriterion} apply.

\begin{proof}[Proof of Corollary~\ref{Cor:CSDLowerOrthant}]
First consider $\gamma=0$. If $Q$ dominates $P$ by FOSD, Strassen's
theorem \citep{Strassen1965} gives a coupling
$\kappa\in\mathcal C(P,Q)$ concentrated on pairs $(x,y)$ with $x\le y$.
Since every $s_z$ is coordinatewise nondecreasing,
$\ell_t(x,y)=0$ whenever $x\le y$. Thus the same coupling satisfies
\eqref{Eq:ScoreCompensatedStrassen} at $\gamma=0$ for every cutoff $t$,
and lower-orthant CSD holds.

Conversely, suppose lower-orthant CSD holds at $\gamma=0$. Choose a finite
$t^\circ>\overline x$ coordinatewise. By
\eqref{Eq:LowerOrthantCSD}, there is
$\kappa\in\mathcal C(P,Q)$ such that
\[
\int_{D\times D}\ell_{t^\circ}(x,y)\,\kappa(dx,dy)=0.
\]
We show that $\ell_{t^\circ}(x,y)>0$ whenever $x\not\le y$.
Choose a coordinate $i$ such that $x_i>y_i$ and consider
\[
B
:=
\left\{
z\in\Theta_{t^\circ}:
y_i\le z_i<x_i,\quad
x_j\vee y_j\le z_j\le t_j^\circ
\text{ for every }j\ne i
\right\}.
\]
Because $t^\circ>\overline x$, the set $B$ has positive Lebesgue measure.
Every $z\in B$ satisfies $y\le z$ but $x\not\le z$. Hence
$s_z(x)=1>s_z(y)=0$ on $B$, so
$\ell_{t^\circ}(x,y)>0$.
Since $\ell_{t^\circ}$ is nonnegative, the zero-integral condition implies
$x\le y$ for $\kappa$-almost every $(x,y)$. Strassen's theorem then gives
FOSD.

Now consider $\gamma=1$. For every $t\ge\underline x$,
Corollary~\ref{Cor:CSDWassersteinCriterion} gives
\[
P\preceq_{s^{(t)},\lambda_t,1}^{\rm CSD}Q
\quad\Longleftrightarrow\quad
\mathbb E_Q[S_t]-\mathbb E_P[S_t]\ge0.
\]
For any probability law $R$ supported on $D$, Fubini's theorem gives
\[
\begin{aligned}
\mathbb E_R
\left[
\prod_{i=1}^d(t_i-X_i)_+
\right]
&=
\mathbb E_R
\left[
\int_{\Theta_t}\mathbf 1_{\{X\le z\}}\,dz
\right]\\
&=
\int_{\Theta_t}F_R(z)\,dz.
\end{aligned}
\]
Using \eqref{Eq:LowerOrthantAggregateScore},
\[
\mathbb E_Q[S_t]-\mathbb E_P[S_t]
=
\int_{\Theta_t}\{F_P(z)-F_Q(z)\}\,dz.
\]
Because $P$ and $Q$ are supported on $D$, for $t\ge\underline x$ the last
integral equals
$\int_{(-\infty,t]}\{F_P(z)-F_Q(z)\}\,dz$.
Thus lower-orthant CSD at $\gamma=1$ is equivalent to
\eqref{Eq:SecondDegreeLowerOrthant} for every $t\ge\underline x$.
If some coordinate of $t$ is below the corresponding coordinate of
$\underline x$, then $F_P(z)=F_Q(z)=0$ for every $z\le t$, so both sides
of \eqref{Eq:SecondDegreeLowerOrthant} are zero. Hence the equivalence
holds for every $t\in\mathbb R^d$.

Finally, let $d=1$ and write $x\wedge t:=\min\{x,t\}$. For
$t\ge\underline x$, \eqref{Eq:LowerOrthantAggregateScore} and
\eqref{Eq:LowerOrthantProfileCost} give
\[
S_t(x)=x\wedge t-\underline x,
\qquad
d_t(x,y)=|x\wedge t-y\wedge t|.
\]
The score profile is determined by $x\wedge t$, and its $L^1(\lambda_t)$
distance is $|x\wedge t-y\wedge t|$. Hence
Lemma~\ref{Lem:ScoreProfileReduction} implies that
$W_{s^{(t)}}(P,Q)$ is the $1$-Wasserstein distance between the laws of
$X\wedge t$ and $Y\wedge t$. The standard one-dimensional formula for
the $1$-Wasserstein distance
\citep{Villani2009OptimalTransport} therefore gives
\[
W_{s^{(t)}}(P,Q)
=
\int_{\underline x}^{t}|F_P(z)-F_Q(z)|\,dz.
\]
The preceding Fubini calculation also gives
\[
\Delta_{s^{(t)}}(P,Q)
=
\int_{\underline x}^{t}\{F_P(z)-F_Q(z)\}\,dz.
\]

Within this proof, define
\[
A_t
:=
\int_{\underline x}^{t}\{F_P(z)-F_Q(z)\}_+\,dz,
\qquad
B_t
:=
\int_{\underline x}^{t}\{F_Q(z)-F_P(z)\}_+\,dz.
\]
Then $W_{s^{(t)}}(P,Q)=A_t+B_t$ and
$\Delta_{s^{(t)}}(P,Q)=A_t-B_t$.
Corollary~\ref{Cor:CSDWassersteinCriterion} therefore gives
\[
P\preceq_{s^{(t)},\lambda_t,\gamma}^{\rm CSD}Q
\quad\Longleftrightarrow\quad
A_t-B_t
\ge
\frac{1-\gamma}{1+\gamma}(A_t+B_t),
\]
which is equivalent to $B_t\le\gamma A_t$. This is exactly
\eqref{Eq:CSDOneDimensionalBetweenOrders}.

For $t<\underline x$, both sides of
\eqref{Eq:CSDOneDimensionalBetweenOrders} are zero. For
$t>\overline x$, $F_P(z)-F_Q(z)=0$ for $z>\overline x$, so the condition
at $t$ is the same as the condition at $\overline x$. Thus requiring
\eqref{Eq:CSDOneDimensionalBetweenOrders} for every finite
$t\ge\underline x$ is equivalent to requiring it for every
$t\in\mathbb R$. The integral characterization in
\citet{MullerScarsiniTsetlinWinkler2017} identifies this condition with
$(1+\gamma)$-stochastic dominance.
\end{proof}

\section{Proofs and Technical Details for Reference-Weighted Stochastic Dominance}
\label{App:RWSDProofs}

\subsection{CVaR representation and proof of Theorem~\ref{Thm:CVAR}}
\label{App:RWSDCVaR}

The proof uses the standard risk-envelope representation of CVaR:
for every bounded measurable $L$ and $C\ge1$,
\begin{equation}
\label{Eq:AppendixCVARRiskEnvelope}
\CVaR_C^\rho(L)
=
\sup_{\substack{\nu\in\mathcal P(\Theta)\\ \nu\le C\rho}}
\int_\Theta L\,d\nu.
\end{equation}
The scalar representation in \eqref{Eq:CVAR} is Theorem~10 and
Equation~(28) in \citet{RockafellarUryasev2002}. The equivalent density-cap
risk envelope in \eqref{Eq:AppendixCVARRiskEnvelope} is standard; see
Section~6.2.4 in \citet{ShapiroDentchevaRuszczynski2021}. We use these established representations in the proof.

\begin{proof}[Proof of Theorem~\ref{Thm:CVAR}]
The assumptions imply that $\Gamma_{P,Q}$ is bounded and
$\mathscr A$-measurable. They also imply that
$x\mapsto\int_\Theta \varphi_\theta(x)\,\Pi(d\theta)$ is bounded and
$\mathscr X$-measurable for every $\Pi\in\mathcal P(\Theta)$. Fix
$P,Q\in\mathcal P(X)$. Every $u\in\mathcal U_{\Phi;k,K;m,M}$ has the form
$u(x)=a+\tau\int_\Theta \varphi_\theta(x)\,\Pi(d\theta)$ for some
$a\in\mathbb R$, $\tau\ge0$, and
$\Pi\in\mathcal P_{k,K;m,M}$. Fubini's theorem gives
\[
\begin{aligned}
\mathbb E_Q[u]-\mathbb E_P[u]
&=
\tau
\left[
\int_X\int_\Theta \varphi_\theta(x)\,\Pi(d\theta)\,Q(dx)
-
\int_X\int_\Theta \varphi_\theta(x)\,\Pi(d\theta)\,P(dx)
\right] \\
&=
\tau
\int_\Theta
\left[
\mathbb E_Q[\varphi_\theta]
-
\mathbb E_P[\varphi_\theta]
\right]\Pi(d\theta) \\
&=
\tau\int_\Theta \Gamma_{P,Q}(\theta)\,\Pi(d\theta).
\end{aligned}
\]
Therefore, $Q\succeq_{\Phi;k,K;m,M}P$ if and only if
\begin{equation}
\label{Eq:WorstReferenceIntegralCondition}
\int_\Theta \Gamma_{P,Q}\,d\Pi\ge0
\qquad
\text{for every }\Pi\in\mathcal P_{k,K;m,M}.
\end{equation}
Indeed, \eqref{Eq:WorstReferenceIntegralCondition} makes the expected-utility
difference nonnegative for every $\tau\ge0$. Conversely, if it fails for an
admissible $\Pi$, then $a=0$ and $\tau=1$ produce an admissible utility with a
negative expected-utility difference.

It remains to evaluate the smallest admissible integral. By
\eqref{Eq:TwoReferenceClass}, every $\Pi\in\mathcal P_{k,K;m,M}$ can be
written as $\Pi=k\,m+\mu$, where $0\le\mu\le K\,M-k\,m$ and
$\mu(\Theta)=1-k$. Writing $\mu=(1-k)\nu$, these conditions are equivalent to
$\nu\in\mathcal P(\Theta)$ and $\nu\le C\,\overline R$, because
\[
(1-k)C\,\overline R
=
(1-k)\frac{K-k}{1-k}\overline R
=
(K-k)\overline R
=
K\,M-k\,m.
\]
Consequently,
\[
\mathcal P_{k,K;m,M}
=
\left\{
k\,m+(1-k)\nu:
\nu\in\mathcal P(\Theta),\ \nu\le C\,\overline R
\right\}.
\]
Hence
\[
\begin{aligned}
\inf_{\Pi\in\mathcal P_{k,K;m,M}}
\int_\Theta \Gamma_{P,Q}\,d\Pi
&=
\inf_{\substack{\nu\in\mathcal P(\Theta)\\ \nu\le C\,\overline R}}
\left\{
k\int_\Theta \Gamma_{P,Q}\,dm
+
(1-k)\int_\Theta \Gamma_{P,Q}\,d\nu
\right\} \\
&=
k\int_\Theta \Gamma_{P,Q}\,dm
+
(1-k)
\inf_{\substack{\nu\in\mathcal P(\Theta)\\ \nu\le C\,\overline R}}
\int_\Theta \Gamma_{P,Q}\,d\nu \\
&=
k\int_\Theta \Gamma_{P,Q}\,dm
-
(1-k)
\sup_{\substack{\nu\in\mathcal P(\Theta)\\ \nu\le C\,\overline R}}
\int_\Theta(-\Gamma_{P,Q})\,d\nu \\
&=
k\int_\Theta \Gamma_{P,Q}\,dm
-
(1-k)\CVaR_C^{\overline R}(-\Gamma_{P,Q}),
\end{aligned}
\]
where the final equality uses \eqref{Eq:AppendixCVARRiskEnvelope}. Condition
\eqref{Eq:WorstReferenceIntegralCondition} is therefore equivalent to
\[
k\int_\Theta \Gamma_{P,Q}\,dm
-
(1-k)\CVaR_C^{\overline R}(-\Gamma_{P,Q})
\ge0.
\]
Multiplying by $-1$ gives \eqref{Eq:TwoReferenceDominanceCVAR}.
\end{proof}

\subsection{Utility derivations for the RWSD examples}
\label{App:RWSDExampleDerivations}

\begin{lemma}[Integral representation of increasing concave utility]
\label{Lem:SlopeMeasure}
Let $u$ be a nonconstant increasing concave function on $[0,\overline x]$ with
$0<u'_+(0)<\infty$. There is a unique probability measure $\Pi_u$ on
$(0,\overline x]$ such that
\begin{equation}
\label{Eq:SlopeMeasureRepresentation}
u(x)=u(0)+u'_+(0)
\int_{(0,\overline x]}\min\{x,t\}\,\Pi_u(dt).
\end{equation}
For $x\in(0,\overline x]$,
\begin{equation}
\label{Eq:SlopeMeasureTail}
\Pi_u([x,\overline x])=
\frac{u'_-(x)}{u'_+(0)}.
\end{equation}
\end{lemma}

\begin{proof}
Set $c:=u'_+(0)$. Concavity and monotonicity imply that
$T(x):=u'_-(x)/c$ is nonincreasing, left-continuous, takes values in $[0,1]$,
and converges to one as $x\downarrow0$. Hence $T$ is the tail function of a
unique probability measure $\Pi_u$ on $(0,\overline x]$, which gives
\eqref{Eq:SlopeMeasureTail}.

Let
$v(x):=\int\min\{x,t\}\,\Pi_u(dt)$. Then
$v'_-(x)=\Pi_u([x,\overline x])=u'_-(x)/c$ for
$x\in(0,\overline x]$. The finite derivative at zero makes $u$ Lipschitz on
the compact interval; $v$ is Lipschitz as well. Thus both functions are
absolutely continuous, and equality of their derivatives almost everywhere,
together with $v(0)=0$, proves \eqref{Eq:SlopeMeasureRepresentation}.
Uniqueness follows from \eqref{Eq:SlopeMeasureTail}.
\end{proof}

For Example~\ref{example:FOSD-RWSD}, an admissible measure has density $\pi$
with $kw_m\le\pi\le Kw_M$ and $\int\pi=1$. The generated utility is
\[
u(x)=a+\tau\int_{\underline x}^{x}\pi(t)\,dt,
\]
so $u'=\tau\pi$ almost everywhere. Dividing by
$\int u'=\tau$ gives \eqref{Eq:RWSDFOSDMarginalUtility}. Conversely, a nonconstant
absolutely continuous increasing utility satisfying \eqref{Eq:RWSDFOSDMarginalUtility}
generates the admissible density $\pi=u'/\int u'$ and is therefore in the
RWSD utility class.

For Example~\ref{example:LowerOrthant-RWSD}, let $q_1,q_2,q_{12}$ denote the
component densities of an admissible integrating measure. Direct integration
of the three test-function families gives
\[
\begin{aligned}
u(x_1,x_2)
=a+\tau\Bigg[1
&-\int_{x_1}^{\overline x_1}q_1(t_1)\,dt_1
-\int_{x_2}^{\overline x_2}q_2(t_2)\,dt_2\\
&-\int_{x_1}^{\overline x_1}\int_{x_2}^{\overline x_2}
q_{12}(t_1,t_2)\,dt_2dt_1\Bigg].
\end{aligned}
\]
Thus, at almost every interior point,
\[
\begin{aligned}
u_1(x)
&=\tau\left[q_1(x_1)+
\int_{x_2}^{\overline x_2}q_{12}(x_1,t_2)\,dt_2\right],\\
u_2(x)
&=\tau\left[q_2(x_2)+
\int_{x_1}^{\overline x_1}q_{12}(t_1,x_2)\,dt_1\right],\\
-u_{12}(x)&=\tau q_{12}(x).
\end{aligned}
\]
Combining these identities with the componentwise reference bounds yields the
marginal-utility bounds described in the main text and
$\tau k w_{m,12}\le-u_{12}\le\tau K w_{M,12}$.

For Example~\ref{example:SOSD-RWSD}, the no-mass condition at $0$ makes every
admissible integrating measure a probability measure on $(0,\overline x]$.
Lemma~\ref{Lem:SlopeMeasure} then shows that the probability measure
$\Pi_u$ in the utility representation is exactly the integrating measure. If $u'_+$ is
absolutely continuous on $(0,\overline x)$, its density is
$-u''(t)/u'_+(0)$ and its possible atom at the upper endpoint is
$u'_-(\overline x)/u'_+(0)$. The two reference inequalities therefore give
exactly the curvature and endpoint bounds stated in the example.

\section{Proof of finite-sample guarantees}

\subsection{Proof of the finite-sample CSD guarantee}
\label{App:CSDStatProofs}

We first note that for empirical distributions, the transportation polytope is fixed and compact,
and its optimal value is a continuous function of the finite cost matrix.
Because every cost-matrix entry is measurable in the observations, the random
optimal values used below are measurable.

\begin{remark}[Bounds required by the CSD confidence statement]
Proposition~\ref{Prop:CSDStatisticalCertificate} needs finite bounds on the
effect of replacing one observation. 
For a known bounded support $D_0$ containing both distributions, let
$\operatorname{diam}_1(D_0)$ be the largest $\ell_1$-distance between two
points in $D_0$. Since $t\mapsto t_+$ is $1$-Lipschitz, each of
$c_P^L,c_Q^L,c_P^S,c_Q^S$ is at most
$\overline L_s\operatorname{diam}_1(D_0)$. For coordinate scores on
$D_0=\prod_{i=1}^d[\underline x_i,\overline x_i]$, the sharper common bound
is $\sum_{i=1}^d\beta_i(\overline x_i-\underline x_i)$.
\end{remark}

\begin{proof}[Proof of Proposition~\ref{Prop:CSDStatisticalCertificate}]
Write $L^\star:=L_s^\star(P,Q)$ and
$\widehat L^\star:=\widehat L_s^\star$. The product coupling of the two
empirical measures gives
\[
\widehat L^\star
\le
\frac{1}{N_P N_Q}
\sum_{i=1}^{N_P}\sum_{j=1}^{N_Q}\ell_s(X_i,Y_j).
\]
By Lemma~\ref{Lem:ScoreRegularity} and the finite first moments, the
right-hand side is integrable. Hence $\widehat L^\star$ is integrable and
the expectations below are well defined. At $\gamma=0$,
Theorem~\ref{Thm:CSDUtilityRepresentation} gives
\[
L^\star
=
\sup_{u\in\mathcal U_{s,\lambda,0}}
\left\{
\int_D u\,dP-
\int_D u\,dQ
\right\}.
\]
The empirical averages are unbiased for each fixed utility. Therefore,
\[
L^\star
\le
\mathbb E\left[
\sup_{u\in\mathcal U_{s,\lambda,0}}
\left\{
\int_D u\,d\widehat P_{N_P}-
\int_D u\,d\widehat Q_{N_Q}
\right\}
\right]
=
\mathbb E[\widehat L^\star].
\]
Thus $F:=L^\star-\widehat L^\star$ satisfies $\mathbb E[F]\le0$, and hence
\begin{equation}
\label{Eq:CSDCenteredLossInclusion}
\{F>t\}
\subseteq
\{F-\mathbb E[F]>t\}
\qquad (t>0).
\end{equation}

For fixed sample values, $\widehat L^\star$ is the value of the empirical
transportation problem with cost $\ell_s$. Replacing $X_i$ by
$X_i'\in D_P$ leaves the transportation polytope unchanged and modifies
only row $i$, whose mass is $1/N_P$. Hence the optimal value changes by at
most $c_P^L/N_P$. Replacing one $Y_j$ changes it by at most
$c_Q^L/N_Q$. If
$(c_P^L)^2/N_P+(c_Q^L)^2/N_Q=0$, then $\widehat L^\star$ is constant on the
full-measure sample domain. Since
$L^\star\le\mathbb E[\widehat L^\star]$, the desired bound follows
directly. Otherwise,
McDiarmid's inequality and \eqref{Eq:CSDCenteredLossInclusion} give
\[
\mathbb P\{F>t\}
\le
\exp\left\{
-\frac{2t^2}
{(c_P^L)^2/N_P+(c_Q^L)^2/N_Q}
\right\}.
\]
Let $r_L:=\overline L_\delta-\widehat L^\star$, which is the
nonnegative square-root term in the definition of $\overline L_\delta$.
Taking $t=r_L$ gives
\[
\mathbb P\{L^\star>\widehat L^\star+r_L\}
\le
\frac{\delta}{2};
\]
see Theorem~6.2 in \citet{BoucheronLugosiMassart2013}. Thus
\begin{equation}
\label{Eq:CSDPopulationLossBound}
L_s^\star(P,Q)\le\overline L_\delta
\end{equation}
with probability at least $1-\delta/2$.

The statistic $\widehat\Delta_s$ is unbiased for $\Delta_s(P,Q)$.
Replacing one $X_i$ changes it by at most $c_P^S/N_P$, and replacing one
$Y_j$ changes it by at most $c_Q^S/N_Q$. Let
$r_\Delta:=\widehat\Delta_s-\underline\Delta_\delta$, the square-root term
in the definition of $\underline\Delta_\delta$. A second application of
bounded differences gives
\[
\mathbb P\{\Delta_s(P,Q)<\widehat\Delta_s-r_\Delta\}
\le
\frac{\delta}{2}.
\]
Equivalently,
\begin{equation}
\label{Eq:CSDPopulationMeanBound}
\Delta_s(P,Q)\ge\underline\Delta_\delta
\end{equation}
with probability at least $1-\delta/2$. A union bound makes
\eqref{Eq:CSDPopulationLossBound} and
\eqref{Eq:CSDPopulationMeanBound} simultaneous with probability at least
$1-\delta$.

On this event, Theorem~\ref{Thm:CompensatedScoreTransport} yields, for every
$\gamma\in[0,1]$,
\[
\mathfrak D_{s,\lambda,\gamma}(P,Q)
=
(1-\gamma)L_s^\star(P,Q)-\gamma\Delta_s(P,Q)
\le
(1-\gamma)\overline L_\delta-
\gamma\underline\Delta_\delta.
\]
This proves \eqref{Eq:CSDUniformGammaBound}. Because the event is uniform in
$\gamma$, the same implication remains valid when $\gamma$ is selected from
the data. If
$\overline{\mathfrak D}_\delta(\widehat\gamma)\le0$, then
$\mathfrak D_{s,\lambda,\widehat\gamma}(P,Q)\le0$ on the simultaneous
event. By Definition~\ref{Def:ScoreCompensatedDominance} and
Theorem~\ref{Thm:CompensatedScoreTransport}, this implies
$P\preceq_{s,\lambda,\widehat\gamma}^{\rm CSD}Q$. The complement of the
simultaneous event has probability at most $\delta$, proving
\eqref{Eq:CSDFalseCertification}.
\end{proof}

\begin{proof}[Proof of Corollary~\ref{Cor:CSDCriticalGammaCertificate}]
On the event in \eqref{Eq:CSDUniformGammaBound}, a finite
$\overline\gamma_\delta$ satisfies
$\mathfrak D_{s,\lambda,\overline\gamma_\delta}(P,Q)\le0$. Hence
$\gamma_{s,\lambda}^\star(P,Q)\le\overline\gamma_\delta$. CSD is nested
in $\gamma$ by Proposition~\ref{Prop:CSDOrderStructure}, so every larger
value also supports the comparison. If
$\overline\gamma_\delta=+\infty$, the coverage statement is automatic.
\end{proof}

\subsection{Proof of the finite-sample RWSD guarantee}
\label{App:RWSDStatProof}

\begin{proof}[Proof of Theorem~\ref{Thm:RWSD_STATS}]
We first verify the measurability assertion used in the theorem. For any
bounded measurable $G:\Theta\to\mathbb R$, the admissible-measure
decomposition in the proof of Theorem~\ref{Thm:CVAR} and the scalar CVaR
representation give
\begin{equation}
\label{Eq:RWSDValueScalarRepresentation}
V(G)
=
k\int_\Theta G\,dm
+(1-k)
\inf_{\eta\in\mathbb R}
\left\{
\eta+C\int_\Theta(G-\eta)_+\,d\overline R
\right\}.
\end{equation}
For $G=-\Gamma_{\widehat P_{n_P},\widehat Q_{n_Q}}$, joint measurability of
$(x,\theta)\mapsto\varphi_\theta(x)$ implies that the integrals in
\eqref{Eq:RWSDValueScalarRepresentation} are measurable functions of the two
samples for each fixed $\eta$. The objective is continuous in $\eta$, so its
infimum over $\mathbb R$ equals its infimum over $\mathbb Q$. Hence the
empirical value is a countable infimum of measurable functions and is
measurable.

Let
$H(\theta):=-\Gamma_{P,Q}(\theta)
+\Gamma_{\widehat P_{n_P},\widehat Q_{n_Q}}(\theta)$. Equivalently,
\[
H(\theta)
=
\left\{
\mathbb E_P[\varphi_\theta]
-
\mathbb E_{\widehat P_{n_P}}[\varphi_\theta]
\right\}
+
\left\{
\mathbb E_{\widehat Q_{n_Q}}[\varphi_\theta]
-
\mathbb E_Q[\varphi_\theta]
\right\}.
\]
The two terms are independent. Since the range of $\varphi_\theta$ has
length at most $b$, Hoeffding's lemma
\citep[Ch.~2]{BoucheronLugosiMassart2013} gives, for every $\xi>0$,
\begin{equation}
\label{Eq:StatHoeffdingMGF}
\mathbb E\exp\{\xi H(\theta)\}
\le
\exp\left\{
\frac{\xi^2 b^2}{8}
\left(\frac1{n_P}+\frac1{n_Q}\right)
\right\}.
\end{equation}

Set $\Pi_0:=k\,m+(1-k)\overline R$. By
\eqref{Eq:ResidualReference}, $\Pi_0$ is a probability measure. It also
belongs to $\mathcal P_{k,K;m,M}$ because $\Pi_0\ge k\,m$ and
\[
\Pi_0
=
k\,m+(1-k)\overline R
\le
k\,m+(K-k)\overline R
=
K\,M.
\]
Integrating \eqref{Eq:StatHoeffdingMGF} with respect to $\Pi_0$ and applying
Fubini--Tonelli gives
\[
\mathbb E
\left[
\int_\Theta \exp\{\xi H(\theta)\}\,\Pi_0(d\theta)
\right]
\le
\exp\left\{
\frac{\xi^2 b^2}{8}
\left(\frac1{n_P}+\frac1{n_Q}\right)
\right\}.
\]
Markov's inequality therefore implies that, with probability at least
$1-\delta$,
\begin{equation}
\label{Eq:StatExponentialEvent}
\log\int_\Theta \exp\{\xi H(\theta)\}\,\Pi_0(d\theta)
\le
\frac{\xi^2 b^2}{8}
\left(\frac1{n_P}+\frac1{n_Q}\right)
+
\log(1/\delta).
\end{equation}

For probability measures $\mu\ll\nu$, write
\[
D_{\mathrm{KL}}(\mu\Vert\nu)
:=
\int \log\left(\frac{d\mu}{d\nu}\right)\,d\mu,
\]
with the convention $0\log0=0$. We next bound the entropy of every admissible
measure relative to $\Pi_0$. Fix $\Pi\in\mathcal P_{k,K;m,M}$. The
decomposition used in the proof of Theorem~\ref{Thm:CVAR} gives
\[
\Pi=k\,m+(1-k)\nu,
\qquad
\nu\in\mathcal P(\Theta),
\qquad
\nu\le C\,\overline R.
\]
Writing $h=d\nu/d\overline R$, we have $0\le h\le C$ and
$\int h\,d\overline R=1$. Hence
\[
D_{\mathrm{KL}}(\nu\Vert \overline R)
=
\int_\Theta h\log h\,d\overline R
\le
\log C.
\]
The same decomposition implies $\Pi\ll\Pi_0$. Applying the log-sum
inequality to Radon--Nikodym densities with respect to a common dominating
measure, as in Theorem~2.7.1 in \citet{CoverThomas2006}, yields
\[
D_{\mathrm{KL}}(\Pi\Vert\Pi_0)
=
D_{\mathrm{KL}}
\bigl(km+(1-k)\nu\,\Vert\,km+(1-k)\overline R\bigr)
\le
(1-k)D_{\mathrm{KL}}(\nu\Vert\overline R)
\le
(1-k)\log C.
\]

The entropy inequality used below follows from Jensen's inequality. Let
$r=d\Pi/d\Pi_0$. Since every admissible measure has finite relative entropy,
\[
\begin{aligned}
\xi\int_\Theta H\,d\Pi-D_{\mathrm{KL}}(\Pi\Vert\Pi_0)
&=
\int_{\{r>0\}}
\log\left(\frac{e^{\xi H}}{r}\right)d\Pi\\
&\le
\log\int_{\{r>0\}}\frac{e^{\xi H}}{r}\,d\Pi
\le
\log\int_\Theta e^{\xi H}\,d\Pi_0.
\end{aligned}
\]
Thus
\begin{equation}
\label{Eq:EntropyInequality}
\xi\int_\Theta H\,d\Pi
\le
D_{\mathrm{KL}}(\Pi\Vert\Pi_0)
+
\log\int_\Theta e^{\xi H}\,d\Pi_0.
\end{equation}
Combining \eqref{Eq:StatExponentialEvent}, \eqref{Eq:EntropyInequality}, and
the entropy bound gives, uniformly over
$\Pi\in\mathcal P_{k,K;m,M}$,
\[
\int_\Theta H\,d\Pi
\le
\frac{\log(1/\delta)+(1-k)\log C}{\xi}
+
\frac{\xi b^2}{8}
\left(\frac1{n_P}+\frac1{n_Q}\right).
\]
If $b=0$, then $H(\theta)=0$ for every $\theta$, so the deviation
bound holds with equality. Suppose $b>0$. Choosing
\[
\xi
=
\sqrt{
\frac{
8\{\log(1/\delta)+(1-k)\log C\}
}{
b^2\left(\frac1{n_P}+\frac1{n_Q}\right)
}
}
\]
yields
\begin{equation}
\label{Eq:StatUniformDeviation}
\sup_{\Pi\in\mathcal P_{k,K;m,M}}
\int_\Theta H\,d\Pi
\le
\varepsilon_{n_P,n_Q}(\delta).
\end{equation}

On the event \eqref{Eq:StatUniformDeviation},
\[
\begin{aligned}
V(-\Gamma_{P,Q})
&=
\sup_{\Pi\in\mathcal P_{k,K;m,M}}
\int_\Theta
\left(
-\Gamma_{\widehat P_{n_P},\widehat Q_{n_Q}}+H
\right)
\,d\Pi \\
&\le
\sup_{\Pi\in\mathcal P_{k,K;m,M}}
\int_\Theta
-\Gamma_{\widehat P_{n_P},\widehat Q_{n_Q}}
\,d\Pi
+
\sup_{\Pi\in\mathcal P_{k,K;m,M}}
\int_\Theta H\,d\Pi \\
&\le
V(-\Gamma_{\widehat P_{n_P},\widehat Q_{n_Q}})
+
\varepsilon_{n_P,n_Q}(\delta).
\end{aligned}
\]
By Definition~\ref{Def:RWSD},
\[
Q\succeq_{\Phi;k,K;m,M}P
\quad\Longleftrightarrow\quad
V(-\Gamma_{P,Q})\le0.
\]
Consequently, on \eqref{Eq:StatUniformDeviation}, the empirical inequality
implies $V(-\Gamma_{P,Q})\le0$ and hence
$Q\succeq_{\Phi;k,K;m,M}P$. The event in
\eqref{Eq:FiniteSampleCertificationError} can therefore occur only when
\eqref{Eq:StatUniformDeviation} fails, which has probability at most
$\delta$.
\end{proof}

\subsection{The exact-order limit}
\label{App:RWSDExactLimit}

\begin{proposition}[Exact-order limit]
\label{Prop:RWSDExactLimit}
Suppose $k=0$, $\Theta$ is a topological space with its Borel
$\sigma$-algebra, $M$ has full support, and $\Gamma_{P,Q}$ is bounded and
continuous. For $K\ge1$, define
$V_K(L):=\sup\{\int_\Theta L\,d\Pi:\Pi\in\mathcal P(\Theta),\ \Pi\le KM\}$.
Then
\[
V_K(-\Gamma_{P,Q})
\uparrow
\sup_{\theta\in\Theta}\{-\Gamma_{P,Q}(\theta)\}
\qquad\text{as }K\to\infty.
\]
\end{proposition}

\begin{proof}
For $k=0$, $V_K(L)=\CVaR_K^M(L)$. The density-cap feasible sets are nested
as $K$ increases, so $K\mapsto\CVaR_K^M(L)$ is nondecreasing. The
risk-envelope form also gives
$\CVaR_K^M(L)\le\operatorname*{ess\,sup}_M L$. For any
$a<\operatorname*{ess\,sup}_M L$, the set $B=\{L>a\}$ has positive
$M$-measure, and the conditional law $M(\cdot\cap B)/M(B)$ is bounded by
$KM$ whenever $K\ge1/M(B)$. Hence $\CVaR_K^M(L)>a$ for all sufficiently
large $K$. It follows that $\CVaR_K^M(L)$ increases to the essential
supremum. Continuity and full support imply that the essential and pointwise
suprema coincide.
\end{proof}

\section{Construction of the Olist Empirical Distributions and Additional Results}
\label{App:OlistApplication}

This appendix explains how we construct the empirical comparison in
Section~\ref{Sec:OlistApplication}. All dominance statements concern the two
weighted empirical distributions defined below. We use $P$ for orders from sellers with lower pre-2018
fulfillment performance and $Q$ for orders from sellers with higher pre-2018
fulfillment performance, according to the rule below.

\subsection{Data files and dates}

We use five tables from the Brazilian E-Commerce Public Dataset by Olist:
orders, order items, order reviews, products, and the product-category
translation table \citep{OlistBrazilianEcommerce}.
The data contain seller deadlines at the item level, delivery dates at the
order level, and review rows linked to orders. Table~\ref{Tab:OlistTiming}
explains the fields used in the analysis.

\begin{table}[t]
\centering
\small
\caption{Olist fields used to classify sellers and measure later service.}
\label{Tab:OlistTiming}
\begin{tabular}{@{}>{\raggedright\arraybackslash}p{0.27\textwidth}
                        >{\raggedright\arraybackslash}p{0.20\textwidth}
                        >{\raggedright\arraybackslash}p{0.45\textwidth}@{}}
\toprule
Field or interval & Data level & Use in the application \\
\midrule
Purchase time & Order & Starts the customer order cycle, determines whether an order enters the 2018 comparison window, and fixes the pre-2018 product-category frequencies. \\
Approval time & Order & Starts the seller handling interval. \\
Shipping deadline & Item & Seller handoff deadline. For an order with several items from one seller, we use the earliest deadline. \\
Carrier-handoff time & Order & Ends seller handling, starts carrier transit, and determines whether a handoff enters the seller's pre-2018 history. \\
Customer-delivery time & Order & Ends carrier transit and the customer's total delivery time. Together with the review answer time, it determines when all eight service outcomes are observed. \\
Estimated-delivery date & Order & Customer-facing date used to measure whether delivery is late and by how much. \\
Review score and answer time & Review row & Measures satisfaction and when the review outcome is observed. When an order has several review rows, we use the latest answered review. \\
\bottomrule
\end{tabular}
\end{table}

We keep an order only when every item has a seller identifier and all items
come from the same seller. Orders with several items from that seller remain
eligible. If several product categories appear in one order, its category is
the category with the largest total item price; alphabetical order breaks a
tie. The elapsed times used for seller handling, carrier transit, and total
delivery must follow the event chronology and be nonnegative. Delivery delay
 is negative when an order arrives before the estimated date.

\subsection{Seller groups and later orders}

The classification date is January 1, 2018. A historical order enters seller
classification when a valid carrier handoff is observed before that date. We
do not condition this historical sample on whether the order is eventually
delivered or reviewed.
For a seller, the first classification measure is the average of
$1-\min\{h/8.4421,1\}$, where $h$ is the number of days from payment approval
to carrier handoff and 8.4421 is the pre-2018 95th percentile. This score is
larger when the seller hands the order to the carrier sooner. The second
measure is the share of handoffs completed no later than the earliest item
shipping deadline. We average the two measures. Sellers need at least ten
observed handoffs to be included. Among 649 eligible sellers, the bottom and top quartiles
contain 163 sellers each.

Separately from the seller classification, we use a pre-2018 calibration
sample to fix the transformations below. An order enters this sample only when
it is a one-seller delivered order and both its delivery and its selected
review answer are observed before January 1, 2018. The 39,819 calibration
orders fix the three time caps and the lower-tail thresholds; no 2018 service
outcome enters these choices.

The later comparison uses one-seller orders purchased from January 1 through
June 30, 2018. An order enters the service comparison when it is delivered,
has the dates needed to construct the eight outcomes below, and has an
answered review. The final samples contain 5,545 orders from 128 sellers in
group $P$ and 6,128 orders from 138 sellers in group $Q$.

We reweight the two groups so that their product-category and calendar-quarter
shares are identical. The 12 most frequent pre-2018 product categories remain separate, all
other categories form a thirteenth group, and these 13 groups are crossed with
the two quarters in the 2018 sample. If $n_{gb}$ is the number of orders from
group $g\in\{P,Q\}$ in category-quarter block $b$, the common mass assigned to
that block is
\[
\omega_b
=
\frac12\left(
\frac{n_{Pb}}{\sum_c n_{Pc}}
+
\frac{n_{Qb}}{\sum_c n_{Qc}}
\right).
\]
Each order from group $g$ and block $b$ receives weight $\omega_b/n_{gb}$.
Consequently, $\widehat P$ and $\widehat Q$ have the same total weight in
every category-quarter block.

\subsection{The eight service outcomes}

Each retained order is represented by eight service outcomes. All coordinates
lie in $[0,1]$ and are oriented so that larger values are preferred.
The first two coordinates measure delivery speed. Let
$T^{\rm carrier}$ be the number of days from carrier handoff to customer
delivery, and let $T^{\rm total}$ be the number of days from purchase to
customer delivery. Define
\[
x_1
=
1-\min\left\{\frac{T^{\rm carrier}}{20.8730},1\right\},
\qquad
x_2
=
1-\min\left\{\frac{T^{\rm total}}{25.7156},1\right\}.
\]
The denominators are the corresponding 95th percentiles in the calibration
sample. A value of
one means that the relevant time is zero; a time at or above the cap receives
zero.

The next three coordinates compare actual delivery with the estimated
delivery date. Let $D$ be customer-delivery time minus estimated-delivery time,
measured in days, so $D\le0$ means early or on time. Define
\[
x_3=\mathbf 1\{D\le0\},
\qquad
x_4=\mathbf 1\{D\le3\},
\qquad
x_5=1-\min\left\{\frac{D_+}{25.6053},1\right\}.
\]
Thus $x_3$ indicates on-time delivery and $x_4$ indicates delivery no more
than three days late. The value $x_5$ equals one for early or on-time delivery
and then decreases linearly with positive delay until it reaches zero at the
cap. The denominator is the 95th percentile of positive delay in the calibration
sample.

The last three coordinates measure customer satisfaction. For an order with
several review rows, we use the review with the latest answer time. The review
identifier is used only when two answer times are equal. If $R\in\{1,\ldots,5\}$ is this final
rating, define
\[
x_6=\frac{R-1}{4},
\qquad
x_7=\mathbf 1\{R\ge3\},
\qquad
x_8=\mathbf 1\{R=5\}.
\]
The first value uses the full 1--5 rating scale, the second distinguishes a
one- or two-star review from a review of at least three stars, and the third indicates a five-star review.

The coordinate comparison assigns weight $1/4$ to each delivery-speed
coordinate and $1/12$ to each of the other six coordinates. Figure~\ref{Fig:OlistScoreMeans} shows the means of the four service scores. Every mean is higher under $\widehat Q$. 

\begin{figure}[t]
\centering
\includegraphics[width=0.88\textwidth]{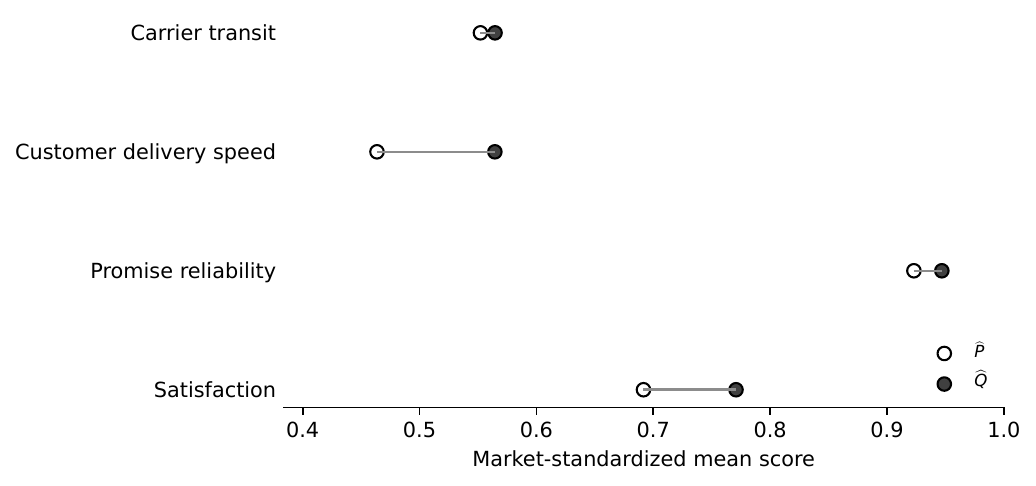}
\caption{Means of the four service scores after reweighting both seller groups
to the same product-category and quarter shares. Open circles denote
$\widehat P$ and filled circles denote $\widehat Q$.}
\label{Fig:OlistScoreMeans}
\end{figure}

\subsection{Four service scores and the lower-tail indicators}

The four scores in \eqref{Eq:OlistFourScores} have weights
$(1/6,1/3,1/4,1/4)$.
For the 75 lower-tail indicators, let $z(x)=(s_1(x),\ldots,s_4(x))$. For
each of the five quantile levels $0.10$, $0.25$, $0.50$, $0.75$, and $0.90$,
and each nonempty subset $A\subseteq\{1,2,3,4\}$, let $q_{rj}$ be the $r$-quantile of score $j$ in the calibration sample and
define
\[
s_{r,A}(x)
=
1-\mathbf 1\{z_j(x)\le q_{rj}\text{ for every }j\in A\}.
\]
Thus $s_{r,A}(x)=0$ on the selected lower-tail event and equals one otherwise.
The weight is $1/\{20\binom{4}{|A|}\}$. Hence each quantile level receives total weight $1/5$. Conditional on a
quantile level, each subset size receives one quarter of that level's weight.

\subsection{Transport computation and complete CSD results}

For each finite collection of scores, we construct 
$d_s(x,y)=\sum_j\lambda_j|s_j(x)-s_j(y)|$ and solve 
\eqref{Eq:FiniteScoreWassersteinLP}. The implementation uses version 0.9.7 of the Python Optimal Transport (POT) package
\citep{FlamaryEtAl2021POT}. The two calculations on the full empirical supports
each contain $5{,}545\times6{,}128=33{,}979{,}760$ transport variables. The
75 binary indicators take only 74 distinct values in group $P$ and 51 in
group $Q$; combining observations with identical 75-indicator vectors reduces that
transport calculation to 3,774 variables.

For the coordinate scores, exact empirical FOSD would require
$W_{1,\beta}(\widehat P,\widehat Q)=
\Delta_\beta(\widehat P,\widehat Q)$. The observed gap between these two
quantities is $0.0000446633$. A separate check of carrier transit gives
$\max_t\{\widehat F_Q(t)-\widehat F_P(t)\}=0.005154$ at $t=0.935950$.
Both calculations show that exact empirical FOSD fails. The minimum expected score
increase reported below is $\{W_s+\Delta_s\}/2$.

\begin{table}[t]
\centering
\small
\caption{Complete CSD quantities for the two reweighted empirical
distributions.}
\label{Tab:OlistFullResults}
\begin{tabular}{@{}lccccc@{}}
\toprule
Scores & $\Delta_s$ & $W_s$ & $L_s^\star$ & Min. increase & $\gamma^\star$ \\
\midrule
Four service scores
& 0.061484 & 0.061514 & 0.00001512 & 0.061499 & 0.00024579 \\
Eight coordinates
& 0.061484 & 0.061528 & 0.00002233 & 0.061506 & 0.00036308 \\
Seventy-five lower-tail indicators
& 0.047662 & 0.047735 & 0.00003680 & 0.047699 & 0.00077155 \\
\bottomrule
\end{tabular}
\end{table}

For the four service scores, Table~\ref{Tab:OlistMovement} separates the
expected increases and decreases under one optimal coupling. The values
already include the score weights. Their column totals are the minimum
expected score increase and $L_s^\star$, respectively.

\begin{table}[!htbp]
\centering
\small
\caption{Expected increases and decreases under one optimal coupling for the
four service scores. Shares use the total increase and total decrease as
separate denominators.}
\label{Tab:OlistMovement}
\begin{tabular}{@{}lrrrr@{}}
\toprule
Score & Increase & Share of increase & Decrease & Share of decrease \\
\midrule
Carrier transit     & 0.002097 & 3.41\%  & 0.00001500 & 99.24\% \\
Customer delivery   & 0.033623 & 54.67\% & 0.00000011 & 0.76\% \\
Delivery date      & 0.005978 & 9.72\%  & 0          & 0 \\
Reviews              & 0.019801 & 32.20\% & 0          & 0 \\
\bottomrule
\end{tabular}
\end{table}

The minimum expected decrease, minimum expected increase, and
$\gamma^\star$ in
Table~\ref{Tab:OlistFullResults} do not depend on which optimal coupling is
selected. The shares in Table~\ref{Tab:OlistMovement} may differ
across optimal couplings, so the table describes the coupling returned by the
solver.

\subsection{Details of the RWSD comparison}

Among the 75 candidate lower-tail indicators, 11 are always zero because all
thresholds used by the indicator equal one. Some remaining definitions
describe the same lower-tail event after constraints with threshold one are
removed. We delete the always-zero candidates, combine the weights of
duplicates, and normalize the remaining weights. This uses the pre-2018
threshold definitions, not the 2018 probability differences. It leaves 23
distinct events with weights $w_j>0$, $j=1,\ldots,23$.

For event $L_j$, the empirical loss is
$\widehat Q(L_j)-\widehat P(L_j)$. We set
$m=M=\sum_{j=1}^{23}w_j\delta_j$, and also use $w$ to denote this
probability measure.
For $0<a<1$, take $k=a$ and $K=1/a$; $a=1$ fixes the weights at $w$.
The allowed weights are
$aw_j\le\pi_j\le w_j/a$ for $j=1,\ldots,23$, and
$\sum_{j=1}^{23}\pi_j=1$. Theorem~\ref{Thm:CVAR} gives,
for $0<a<1$,
\[
V_a
=
a\sum_{j=1}^{23}w_j\{\widehat Q(L_j)-\widehat P(L_j)\}
+
(1-a)\CVaR_{(1+a)/a}^{w}
\bigl((\widehat Q(L_j)-\widehat P(L_j))_{j=1}^{23}\bigr).
\]
The implementation computes the same value by sorting the event losses and
assigning the available weight to the largest losses, as in
\eqref{Eq:FiniteWorstLoss}.

\begin{table}[!htbp]
\centering
\small
\caption{RWSD results for the selected lower-tail events. A positive event
loss means that $\widehat Q$ places more probability in that lower-tail
region.}
\label{Tab:OlistRWSDResults}
\begin{tabular}{@{}lr@{}}
\toprule
Quantity & Value \\
\midrule
Candidate lower-tail indicators & 75 \\
Indicators that are always zero & 11 \\
Distinct remaining events & 23 \\
Events with positive loss & 1 \\
Loss under the reference weights & $-0.054471$ \\
Largest event loss & $0.000883$ \\
Smallest $a$ for which RWSD holds & $0.015956$ \\
Weight on the unfavorable event at that value & $0.9848$ \\
\bottomrule
\end{tabular}
\end{table}

The single positive loss is the event that carrier-transit speed is at most
0.899180, its pre-2018 90th-percentile threshold. Its probability is 0.869046
under $\widehat P$ and 0.869930 under $\widehat Q$. At
$a^\star=0.015956$, the maximizing weights assign $98.48\%$ to this event and
the minimum permitted weight to the others.

%The replication code includes robustness checks and also show that we get similar results by varying the carrier share or other weights. 

\paragraph{Data and code availability.}
The replication package contains the Python and Google Colab code, the five
Olist data files used here, exact file hashes, the tables
and figures. The data are distributed
as the Brazilian E-Commerce Public Dataset by Olist under the Creative Commons
Attribution--NonCommercial--ShareAlike 4.0 license
\citep{OlistBrazilianEcommerce}.

\bibliographystyle{abbrvnat}
\bibliography{CSD_OR}

\begin{thebibliography}{39}
\providecommand{\natexlab}[1]{#1}
\providecommand{\url}[1]{\texttt{#1}}
\expandafter\ifx\csname urlstyle\endcsname\relax
  \providecommand{\doi}[1]{doi: #1}\else
  \providecommand{\doi}{doi: \begingroup \urlstyle{rm}\Url}\fi

\bibitem[Armbruster and Delage(2015)]{ArmbrusterDelage2015}
B.~Armbruster and E.~Delage.
\newblock Decision making under uncertainty when preference information is incomplete.
\newblock \emph{Management Science}, 61\penalty0 (1):\penalty0 111--128, 2015.
\newblock \doi{10.1287/mnsc.2014.2059}.

\bibitem[Ba{\'{\i}}llo et~al.(2025)Ba{\'{\i}}llo, C{\'a}rcamo, and Mora-Corral]{BailloCarcamoMoraCorral2025}
A.~Ba{\'{\i}}llo, J.~C{\'a}rcamo, and C.~Mora-Corral.
\newblock Tests for almost stochastic dominance.
\newblock \emph{Journal of Business \& Economic Statistics}, 43\penalty0 (2):\penalty0 338--350, 2025.
\newblock \doi{10.1080/07350015.2024.2374274}.

\bibitem[Barrett and Donald(2003)]{BarrettDonald2003}
G.~F. Barrett and S.~G. Donald.
\newblock Consistent tests for stochastic dominance.
\newblock \emph{Econometrica}, 71\penalty0 (1):\penalty0 71--104, 2003.
\newblock \doi{10.1111/1468-0262.00390}.

\bibitem[Bauer et~al.(2025)Bauer, M{\'e}moli, Needham, and Nishino]{BauerMemoliNeedhamNishino2025}
M.~Bauer, F.~M{\'e}moli, T.~Needham, and M.~Nishino.
\newblock The {$Z$}-gromov-wasserstein distance.
\newblock \emph{Journal of Machine Learning Research}, 26\penalty0 (291):\penalty0 1--57, 2025.

\bibitem[Boucheron et~al.(2013)Boucheron, Lugosi, and Massart]{BoucheronLugosiMassart2013}
S.~Boucheron, G.~Lugosi, and P.~Massart.
\newblock \emph{Concentration Inequalities: A Nonasymptotic Theory of Independence}.
\newblock Oxford University Press, Oxford, UK, 2013.
\newblock ISBN 9780199535255.
\newblock \doi{10.1093/acprof:oso/9780199535255.001.0001}.

\bibitem[Cover and Thomas(2006)]{CoverThomas2006}
T.~M. Cover and J.~A. Thomas.
\newblock \emph{Elements of Information Theory}.
\newblock John Wiley \& Sons, Hoboken, NJ, 2 edition, 2006.
\newblock ISBN 9780471241959.
\newblock \doi{10.1002/047174882X}.

\bibitem[Davidson and Duclos(2000)]{DavidsonDuclos2000}
R.~Davidson and J.-Y. Duclos.
\newblock Statistical inference for stochastic dominance and for the measurement of poverty and inequality.
\newblock \emph{Econometrica}, 68\penalty0 (6):\penalty0 1435--1464, 2000.
\newblock \doi{10.1111/1468-0262.00167}.

\bibitem[del Barrio et~al.(2018)del Barrio, Cuesta-Albertos, and Matr{\'a}n]{DelBarrioCuestaAlbertosMatran2018}
E.~del Barrio, J.~A. Cuesta-Albertos, and C.~Matr{\'a}n.
\newblock An optimal transportation approach for assessing almost stochastic order.
\newblock In E.~Gil, E.~Gil, J.~Gil, and M.~{\'A}. Gil, editors, \emph{The Mathematics of the Uncertain: A Tribute to Pedro Gil}, volume 142 of \emph{Studies in Systems, Decision and Control}, pages 33--44. Springer International Publishing, Cham, Switzerland, 2018.
\newblock \doi{10.1007/978-3-319-73848-2_3}.

\bibitem[Dentcheva and Yi(2025)]{DentchevaYi2025}
D.~Dentcheva and Y.~Yi.
\newblock Relaxation of stochastic dominance constraints via optimal mass transport.
\newblock \emph{SIAM Journal on Optimization}, 35\penalty0 (4):\penalty0 2452--2473, 2025.
\newblock \doi{10.1137/24M172216X}.

\bibitem[Duclos et~al.(2006)Duclos, Sahn, and Younger]{DuclosSahnYounger2006}
J.-Y. Duclos, D.~Sahn, and S.~D. Younger.
\newblock Robust multidimensional spatial poverty comparisons in ghana, madagascar, and uganda.
\newblock \emph{The World Bank Economic Review}, 20\penalty0 (1):\penalty0 91--113, 2006.
\newblock \doi{10.1093/wber/lhj005}.

\bibitem[Flamary et~al.(2021)Flamary, Courty, Gramfort, Alaya, Boisbunon, Chambon, Chapel, Corenflos, Fatras, Fournier, Gautheron, Gayraud, Janati, Rakotomamonjy, Redko, Rolet, Schutz, Seguy, Sutherland, Tavenard, Tong, and Vayer]{FlamaryEtAl2021POT}
R.~Flamary, N.~Courty, A.~Gramfort, M.~Z. Alaya, A.~Boisbunon, S.~Chambon, L.~Chapel, A.~Corenflos, K.~Fatras, N.~Fournier, L.~Gautheron, N.~T.~H. Gayraud, H.~Janati, A.~Rakotomamonjy, I.~Redko, A.~Rolet, A.~Schutz, V.~Seguy, D.~J. Sutherland, R.~Tavenard, A.~Tong, and T.~Vayer.
\newblock {POT}: Python optimal transport.
\newblock \emph{Journal of Machine Learning Research}, 22\penalty0 (78):\penalty0 1--8, 2021.
\newblock URL \url{https://www.jmlr.org/papers/v22/20-451.html}.

\bibitem[Hu and Stepanyan(2017)]{HuStepanyan2017}
J.~Hu and G.~Stepanyan.
\newblock Optimization with reference-based robust preference constraints.
\newblock \emph{SIAM Journal on Optimization}, 27\penalty0 (4):\penalty0 2230--2257, 2017.
\newblock \doi{10.1137/16M1105050}.

\bibitem[Hu et~al.(2014)Hu, Homem-de Mello, and Mehrotra]{HuHomemDeMelloMehrotra2014}
J.~Hu, T.~Homem-de Mello, and S.~Mehrotra.
\newblock Stochastically weighted stochastic dominance concepts with an application in capital budgeting.
\newblock \emph{European Journal of Operational Research}, 232\penalty0 (3):\penalty0 572--583, 2014.
\newblock \doi{10.1016/j.ejor.2013.08.007}.

\bibitem[Junov{\'a} and Kopa(2025)]{JunovaKopa2025}
J.~Junov{\'a} and M.~Kopa.
\newblock Measures of stochastic non-dominance in portfolio optimization.
\newblock \emph{European Journal of Operational Research}, 321\penalty0 (1):\penalty0 269--283, 2025.
\newblock \doi{10.1016/j.ejor.2024.08.029}.

\bibitem[Kamae et~al.(1977)Kamae, Krengel, and O'Brien]{KamaeKrengelOBrien1977}
T.~Kamae, U.~Krengel, and G.~L. O'Brien.
\newblock Stochastic inequalities on partially ordered spaces.
\newblock \emph{The Annals of Probability}, 5\penalty0 (6):\penalty0 899--912, 1977.
\newblock \doi{10.1214/aop/1176995659}.

\bibitem[Kamihigashi and Stachurski(2020)]{KamihigashiStachurski2020}
T.~Kamihigashi and J.~Stachurski.
\newblock Partial stochastic dominance via optimal transport.
\newblock \emph{Operations Research Letters}, 48\penalty0 (5):\penalty0 584--586, 2020.
\newblock \doi{10.1016/j.orl.2020.07.003}.

\bibitem[Leshno and Levy(2002)]{LeshnoLevy2002}
M.~Leshno and H.~Levy.
\newblock Preferred by ``all'' and preferred by ``most'' decision makers: Almost stochastic dominance.
\newblock \emph{Management Science}, 48\penalty0 (8):\penalty0 1074--1085, 2002.
\newblock \doi{10.1287/mnsc.48.8.1074.169}.

\bibitem[Light and Perlroth(2021)]{LightPerlroth2021}
B.~Light and A.~Perlroth.
\newblock The family of {$\alpha,[a,b]$} stochastic orders: Risk vs.\ expected value.
\newblock \emph{Journal of Mathematical Economics}, 96:\penalty0 102520, 2021.
\newblock \doi{10.1016/j.jmateco.2021.102520}.

\bibitem[Linton et~al.(2005)Linton, Maasoumi, and Whang]{LintonMaasoumiWhang2005}
O.~Linton, E.~Maasoumi, and Y.-J. Whang.
\newblock Consistent testing for stochastic dominance under general sampling schemes.
\newblock \emph{The Review of Economic Studies}, 72\penalty0 (3):\penalty0 735--765, 2005.
\newblock \doi{10.1111/j.1467-937X.2005.00350.x}.

\bibitem[Lizyayev and Ruszczy{\'n}ski(2012)]{LizyayevRuszczynski2012}
A.~Lizyayev and A.~Ruszczy{\'n}ski.
\newblock Tractable almost stochastic dominance.
\newblock \emph{European Journal of Operational Research}, 218\penalty0 (2):\penalty0 448--455, 2012.
\newblock \doi{10.1016/j.ejor.2011.11.019}.

\bibitem[Luo et~al.(2024)Luo, Chen, and Jaillet]{LuoChenJaillet2025}
C.~Luo, P.~Chen, and P.~Jaillet.
\newblock Portfolio optimization based on almost second-degree stochastic dominance.
\newblock \emph{Management Science}, 71\penalty0 (8):\penalty0 7029--7055, 2024.
\newblock \doi{10.1287/mnsc.2022.01092}.

\bibitem[M{\"u}ller(1997)]{Muller1997StochasticOrders}
A.~M{\"u}ller.
\newblock Stochastic orders generated by integrals: A unified study.
\newblock \emph{Advances in Applied Probability}, 29\penalty0 (2):\penalty0 414--428, 1997.
\newblock \doi{10.2307/1428010}.

\bibitem[M{\"u}ller(2024)]{Muller2024StochasticOrders}
A.~M{\"u}ller.
\newblock Stochastic orders under uncertainty.
\newblock In J.~Ansari, S.~Fuchs, W.~Trutschnig, M.~A. Lubiano, M.~{\'A}. Gil, P.~Grzegorzewski, and O.~Hryniewicz, editors, \emph{Combining, Modelling and Analyzing Imprecision, Randomness and Dependence}, volume 1458 of \emph{Advances in Intelligent Systems and Computing}, pages 302--308, Cham, 2024. Springer.
\newblock \doi{10.1007/978-3-031-65993-5_37}.

\bibitem[M{\"u}ller and Stoyan(2002)]{MullerStoyan2002}
A.~M{\"u}ller and D.~Stoyan.
\newblock \emph{Comparison Methods for Stochastic Models and Risks}.
\newblock John Wiley \& Sons, Chichester, UK, 2002.
\newblock ISBN 9780471494461.

\bibitem[M{\"u}ller and Wiesel(2026)]{MullerWiesel2026}
A.~M{\"u}ller and J.~Wiesel.
\newblock Almost stochastic dominance via optimal transport.
\newblock arXiv preprint arXiv:2607.28215, 2026.

\bibitem[M{\"u}ller et~al.(2016)M{\"u}ller, Scarsini, Tsetlin, and Winkler]{MullerScarsiniTsetlinWinkler2017}
A.~M{\"u}ller, M.~Scarsini, I.~Tsetlin, and R.~L. Winkler.
\newblock Between first- and second-order stochastic dominance.
\newblock \emph{Management Science}, 63\penalty0 (9):\penalty0 2933--2947, 2016.
\newblock \doi{10.1287/mnsc.2016.2486}.

\bibitem[M{\"u}ller et~al.(2025)M{\"u}ller, Scarsini, Tsetlin, and Winkler]{MullerScarsiniTsetlinWinkler2025}
A.~M{\"u}ller, M.~Scarsini, I.~Tsetlin, and R.~L. Winkler.
\newblock Multivariate almost stochastic dominance: Transfer characterizations and sufficient conditions under dependence uncertainty.
\newblock \emph{Operations Research}, 73\penalty0 (2):\penalty0 879--893, 2025.
\newblock \doi{10.1287/opre.2022.0596}.

\bibitem[O'Brien and Scarsini(1991)]{OBrienScarsini1991}
G.~L. O'Brien and M.~Scarsini.
\newblock Multivariate stochastic dominance and moments.
\newblock \emph{Mathematics of Operations Research}, 16\penalty0 (2):\penalty0 382--389, 1991.
\newblock \doi{10.1287/moor.16.2.382}.

\bibitem[{Olist}(2018)]{OlistBrazilianEcommerce}
{Olist}.
\newblock Brazilian e-commerce public dataset by olist.
\newblock \url{https://www.kaggle.com/datasets/olistbr/brazilian-ecommerce}, 2018.
\newblock Accessed August 28, 2026.

\bibitem[Rioux et~al.(2024)Rioux, Nitsure, Rigotti, Greenewald, and Mroueh]{RiouxNitsureRigottiGreenewaldMroueh2024}
G.~Rioux, A.~Nitsure, M.~Rigotti, K.~Greenewald, and Y.~Mroueh.
\newblock Multivariate stochastic dominance via optimal transport and applications to models benchmarking.
\newblock In \emph{Advances in Neural Information Processing Systems}, volume~37, pages 39190--39223. Curran Associates, Inc., 2024.
\newblock \doi{10.52202/079017-1237}.
\newblock URL \url{https://proceedings.neurips.cc/paper_files/paper/2024/hash/4537592f9594a0522da99566b90380cc-Abstract-Conference.html}.

\bibitem[Rockafellar and Uryasev(2002)]{RockafellarUryasev2002}
R.~T. Rockafellar and S.~Uryasev.
\newblock Conditional value-at-risk for general loss distributions.
\newblock \emph{Journal of Banking \& Finance}, 26\penalty0 (7):\penalty0 1443--1471, 2002.
\newblock \doi{10.1016/S0378-4266(02)00271-6}.

\bibitem[Shaked and Shanthikumar(2007)]{ShakedShanthikumar2007}
M.~Shaked and J.~G. Shanthikumar.
\newblock \emph{Stochastic Orders}.
\newblock Springer Series in Statistics. Springer, New York, NY, 2007.
\newblock ISBN 9780387329154.
\newblock \doi{10.1007/978-0-387-34675-5}.

\bibitem[Shapiro et~al.(2021)Shapiro, Dentcheva, and Ruszczy{\'n}ski]{ShapiroDentchevaRuszczynski2021}
A.~Shapiro, D.~Dentcheva, and A.~Ruszczy{\'n}ski.
\newblock \emph{Lectures on Stochastic Programming: Modeling and Theory}, volume~28 of \emph{MOS-SIAM Series on Optimization}.
\newblock Society for Industrial and Applied Mathematics and the Mathematical Optimization Society, Philadelphia, PA, 3 edition, 2021.
\newblock ISBN 9781611976588.
\newblock \doi{10.1137/1.9781611976595}.

\bibitem[Song and Sun(2026)]{SongSun2026}
X.~Song and Z.~Sun.
\newblock Almost dominance: Inference and application.
\newblock \emph{Econometric Reviews}, 45\penalty0 (2):\penalty0 283--302, 2026.
\newblock \doi{10.1080/07474938.2025.2565642}.

\bibitem[Strassen(1965)]{Strassen1965}
V.~Strassen.
\newblock The existence of probability measures with given marginals.
\newblock \emph{The Annals of Mathematical Statistics}, 36\penalty0 (2):\penalty0 423--439, 1965.
\newblock \doi{10.1214/aoms/1177700153}.

\bibitem[Tan(2015)]{Tan2015WeightedASD}
C.~H. Tan.
\newblock Weighted almost stochastic dominance: Revealing the preferences of most decision makers in the {St. Petersburg} paradox.
\newblock \emph{Decision Analysis}, 12\penalty0 (2):\penalty0 74--80, 2015.
\newblock \doi{10.1287/deca.2014.0310}.

\bibitem[Tsetlin and Winkler(2018)]{TsetlinWinkler2018}
I.~Tsetlin and R.~L. Winkler.
\newblock Multivariate almost stochastic dominance.
\newblock \emph{Journal of Risk and Insurance}, 85\penalty0 (2):\penalty0 431--445, 2018.
\newblock \doi{10.1111/jori.12222}.

\bibitem[Tsetlin et~al.(2015)Tsetlin, Winkler, Huang, and Tzeng]{TsetlinWinklerHuangTzeng2015}
I.~Tsetlin, R.~L. Winkler, R.~J. Huang, and L.~Y. Tzeng.
\newblock Generalized almost stochastic dominance.
\newblock \emph{Operations Research}, 63\penalty0 (2):\penalty0 363--377, 2015.
\newblock \doi{10.1287/opre.2014.1340}.

\bibitem[Villani(2009)]{Villani2009OptimalTransport}
C.~Villani.
\newblock \emph{Optimal Transport: Old and New}, volume 338 of \emph{Grundlehren der mathematischen Wissenschaften}.
\newblock Springer, Berlin, Germany, 2009.
\newblock ISBN 9783540710509.
\newblock \doi{10.1007/978-3-540-71050-9}.

\end{thebibliography}

\end{document}